\documentclass[12pt,letterpaper]{article}
\usepackage{graphics,epsfig,graphicx,color}
\usepackage{authblk}

\graphicspath{{/EPSF/}{../Figures/}{Figures/}}

\usepackage{amssymb,latexsym}

\usepackage{setspace}
\usepackage{amsmath}

\usepackage{mathrsfs,amsfonts}

\usepackage{amsthm,amsxtra}

\newtheorem{lemma}{Lemma}[section]

\newtheorem{prop}{Proposition}[section]
\newtheorem{corollary}{Corollary}[section]

\newtheorem{hypothesis}{Hypothesis}
\newcommand{\bfm}[1]{\mbox{\boldmath ${#1}$}}

\newcommand{\beqa}{\begin{eqnarray}}
\newcommand{\eeqa}[1]{\label{#1}\end{eqnarray}}
\newcommand{\beq}{\begin{equation}}
\newcommand{\eeq}[1]{\label{#1}\end{equation}}
\newcommand{\Gd}{\delta}
\newcommand{\Ge}{\epsilon}

\newcommand{\BGv}{\bfm\nu}

\newcommand{\re}{\textrm{Re}}

\newcommand{\paren}[1]{\left ( #1\right )}

\def\Bx{{\bf x}}
\def\By{{\bf y}}
\def\Bz{{\bf z}}

\def\B0{{\bf 0}}

\def \RR {{\mathbb R}}

\def \ba {\begin{array}}
\def \ea {\end{array}}
\newtheorem {Thm} {Theorem} [section]

\newtheorem {Adef} [Thm] {Definition}

\newtheorem {Arem} [Thm] {Remark}

\newtheorem {Aexa} [Thm] {Example}

\newtheorem {Anot} [Thm] {Notation}

\def \refe #1.{(\ref{#1})}
\def \reff #1.{figure~\ref{#1}}
\def \refs #1.{section~\ref{#1}}
\def \refss #1.{subsection~\ref{#1}}
\def \refD #1.{Definition~\ref{#1}}
\def \refT #1.{Theorem~\ref{#1}}
\def \refL #1.{Lemma~\ref{#1}}
\def \refC #1.{Corollary~\ref{#1}}
\def \refP #1.{Proposition~\ref{#1}}
\def \refR #1.{Remark~\ref{#1}}
\def \refE #1.{Example~\ref{#1}}
\def \refN #1.{Notation~\ref{#1}}
\newif\ifPDF
\ifx\pdfoutput\undefined
\PDFfalse
\else
\ifnum\pdfoutput > 0
\PDFtrue
\else
\PDFfalse
\fi
\fi

\ifPDF
\usepackage{pdftricks}
\begin{psinputs}
\usepackage{pstricks}
\usepackage{pstcol}
\usepackage{pst-plot}
\usepackage{pst-tree}
\usepackage{pst-eps}
\usepackage{multido}
\usepackage{pst-node}
\usepackage{pst-eps}
\end{psinputs}
\else
\usepackage{pstricks}
\fi

\ifPDF
\usepackage[debug,pdftex,colorlinks=true,
linkcolor=blue, bookmarksopen=false,
plainpages=false,pdfpagelabels]{hyperref}
\else
\usepackage[dvips]{hyperref}
\fi

\usepackage{fancyhdr}

\title{Sensitivity analysis for active control of the Helmholtz equation}

\author[1]{Mark Hubenthal}
\author[1]{Daniel Onofrei}		      
\affil[1]{Department of Mathematics, University of Houston, Houston, Texas 77004}

\begin{document}
\maketitle

\begin{abstract}
The results in \cite{O2} (see \cite{O1} for the quasistatics regime) consider the Helmholtz equation with fixed frequency $k$ and, in particular imply that, for $k$ outside a discrete set of resonant frequencies and given a source region $D_a\subset \RR^d$ ($d=\overline{2,3}$) and $u_0$, a solution of the homogeneous scalar Helmholtz equation in a set containing the control region $D_c\subset \RR^d$, there exists an infinite class of boundary data on $\partial D_a$ so that the radiating solution to the corresponding exterior scalar Helmholtz problem in $\RR^d\setminus D_a$ will closely approximate $u_0$ in $D_c$. Moreover, it will have vanishingly small values beyond a certain large enough ``far-field" radius $R$ (see Figure \ref{fig:mainsetup} for a geometric description).

In this paper we study the minimal energy solution of the above problem (e.g. the solution obtained by using Tikhonov regularization with the Morozov discrepancy principle) and perform a detailed sensitivity analysis. In this regard we discuss the stability of the the minimal energy solution with respect to measurement errors as well as the feasibility of the active scheme (power budget and accuracy) depending on: the mutual distances between the antenna, control region and far field radius $R$, value of regularization parameter, frequency, location of the source.
\end{abstract}

\section{Introduction}
During recent years, there has been a growing interest in the development of feasible strategies for the control of acoustic and electromagnetic fields with one possible application being the construction of robust schemes for sonar or radar cloaking.

One main approach controls fields in the regions of interest by changing the material properties of the medium in certain surrounding regions (\cite{Chan3,Chan2,Cummer,Green2,Green1,Gunther-pr,Pendry} and references therein). Several alternative techniques are proposed in the literature (other than transformation optics strategies) such as: plasmonic designs (see \cite{Alu} and references therein), strategies based on anomalous resonance phenomena (see \cite{Mil1,Mil3,Mil2}), conformal mapping techniques (see \cite{Ulf2,Ulf1}), and complementary media strategies (see \cite{Chan}).

In the applied community, active designs for the manipulation of fields appear to have occurred initially in the context of low-frequency acoustics (or active noise cancellation). Especially notable are the pioneering works of Lueg \cite{Lueg} (feed-forward control of sound) and Olson \& May \cite{Olson-May} (feedback control of sound). The reviews \cite{Elliot,Fuller,Tsynkov,T1,Peake,Peterson}, provide detailed accounts of past and recent developments in acoustic active control.

In the context of cloaking, the {\bf{\textit{interior}}} strategy proposed in \cite{Miller} employs a continuous active layer on the boundary of the control region while the {\bf{\textit{exterior}}} scheme discussed in \cite{OMV4,OMV1,OMV2,OMV3} (see also \cite{CTchan-num}), uses a discrete number of active sources located in the exterior of the control region to manipulate the fields. The active exterior strategy for 2D quasistatics cloaking was introduced in \cite{OMV1}, and, based on \emph{a priori} information about the incoming field, the authors constructively described how one can create an almost zero field control region with very small effect in the far field. However, the proposed strategy did not work for control regions close to the active source. It ``cloaked" large objects only when they are far enough from the source region (see \cite{OMV4}) and was not adaptable to three space dimensions. The finite frequency case was studied in the last section of \cite{OMV1} and in \cite{OMV3} (see also \cite{OMV4} for a recent review) where three (or four in 3D) active sources were needed to create a zero field region in the interior of their convex hull, while creating a very small scattering effect in the far field. The broadband character of the proposed scheme was numerically observed in \cite{OMV2}. All the above results were obtained assuming large amplitude and highly oscillatory currents on the active source regions. In this regard, in \cite{Norris} (see also \cite{Devaney0,Miller}) the authors presented theoretical and numerical evidence that increasing the number of sources will decrease the power needed on each source and thus increase the feasibility of the scheme. Experimental designs and testing of active cloaking schemes in various regimes are reported in \cite{Du,Ma,Eleftheriades1,Eleftheriades2}.

In a recent development in \cite{O1}, a general analytical approach based on the theory of boundary layer potentials is proposed for the active control problem in the quasi-static regime.  By using the same integral equation approach, in \cite{O2} we extended the results presented in \cite{O1} to the active control problem for the exterior scalar Helmholtz equation. In particular, we characterized an infinite class of boundary functions on the source boundary $\partial D_a$ so that we achieve the desired manipulation effects in several mutually disjoint exterior regions. The method is novel in the sense that instead of using microstructures, exterior active sources modeled with the help of the above boundary controls are employed for the desired control effects. Such exterior active sources can represent velocity potential, pressure or currents.

In the current paper we study the active control problem in the context of cloaking, where one antenna $D_a$ protects a given control region $D_c$ from far field interrogation on $\partial B_{R}(\mathbf{0})$, with $R \gg 1$ (see Figure \ref{fig:mainsetup}). We make use of the results in \cite{O2} and present a detailed sensitivity and feasibility study for the minimal norm solution of the problem.

The paper is organized as follows: In Section \ref{sec:background} we recall the general result obtained in \cite{O2} in the context of exterior active cloaking. In Section \ref{sec:stability} we present an $L^2$ conditional stability result for the minimal norm solution with respect to measurement errors of the incoming field. In Section \ref{sec:numerics} we present the numerical details of the Tikhonov regularization algorithm with the Morozov discrepancy principle for the computation of the minimal norm solution of the exterior active cloaking problem in two dimensions. We will numerically observe the fact that the scheme requires large antenna powers in the far field and we will provide numerical support for our theoretical stability results. An important part of this section will be focused on the sensitivity analysis, where we will study: the dependence of the control results as a function of mutual distances between the antenna, control region and far field region; and the broadband character of our scheme in the near field region. Finally, in Section \ref{sec:conclusions} we highlight the main results of the paper and discuss current and future challenges and extensions of our research.	

\section{Background}
\label{sec:background}
\begin{figure}
\centering
\def \svgwidth{0.4\linewidth}
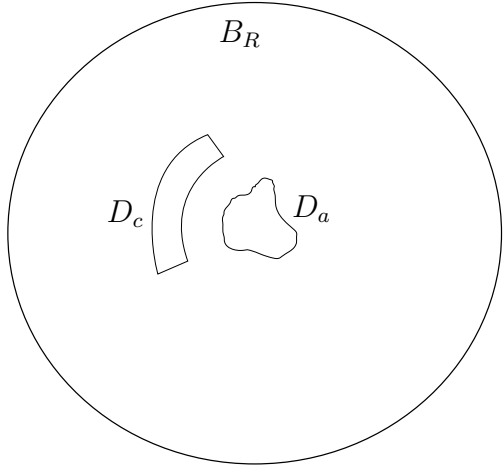
\caption{An antenna defined by $\partial D_{a}$ with a control region $D_{c}$ and far field region $B_R(\B0)$. }\label{fig:mainsetup}
\end{figure}
In this section we will recall the main result regarding the active exterior control problem for the Helmholtz equation obtained in \cite{O2}. We will focus only on the case where one active external source (antenna) $D_a$ protects a control region $D_c$ from an interrogating far field and maintains an overall small signature beyond a disk of large enough radius $R$.

The general setup for this question will be as follows. Let $B_{R} \subset \mathbb{R}^{d}$ be the ball of radius $R > 0$. We assume $\B0\in D_{a} \subset  B_{R}$ is the region inside a single antenna with sufficiently smooth boundary $\partial D_{a}$. We also let $D_{c} \Subset B_{R}$ be the control region, which is assumed to satisfy $\overline{D_{c}} \cap \overline{D_{a}} = \emptyset$ (see Figure \ref{fig:mainsetup}). The numerical simulations in the current work are performed for the two dimensional case but the methods are adaptable to the three dimensional setting as well.
Consider the function space
\begin{equation*}
\Xi = L^{2}(\partial D_{c}) \times L^{2}(\partial B_{R}),
\end{equation*}
endowed with the scalar product
\begin{equation}
(\phi,\psi)_{\Xi} = \int_{\partial D_{c}}\phi_{1}(\mathbf{y})\overline{\psi}_{1}(\mathbf{y})\,dS_{\mathbf{y}} + \int_{\partial B_{R}} \phi_{2}(\mathbf{y})\overline{\psi}_{2}(\mathbf{y})\,dS_{\mathbf{y}},
\end{equation}
which is a Hilbert space. For the remainder of the paper we will assume that every $L^2$ space of complex valued functions will be endowed with the usual inner product. As in \cite{O2} consider $K: L^{2}(\partial D_{a}) \to \Xi$, the double layer potential operator restricted to $\partial D_{c}$ and $\partial B_{R}$, respectively, defined by
\begin{equation}
\label{EQ:K}
K\phi(\Bx,\Bz) =  (K_{1}\phi(\Bx), K_{2}\phi(\Bz)), \quad \phi \in L^{2}(\partial D_{a}),
\end{equation}
where
\begin{align}
K_1:L^{2}(\partial D_{a})\rightarrow L^{2}(\partial D_{c}), \, K_1\phi(\Bx) & = \int_{\partial D_{a}}\phi(\By)\frac{\partial \Phi(\Bx,\By)}{\partial \BGv_{\By}}ds_{\By}, \mbox{ for }\Bx\in
\partial D_c, \notag\\
K_2:L^{2}(\partial D_{a})\rightarrow L^{2}(\partial B_R), \, K_2\phi(\Bz) & = \int_{\partial D_{a}}\phi(\By)\frac{\partial \Phi(\Bz,\By)}{\partial \BGv_{\By}}ds_{\By}, \mbox{ for }\Bz\in
\partial B_R(\B0). \label{8}
\end{align}
Here $\Phi(\Bx, \By)$ represents the fundamental solution of the relevant Helmholtz operator, i.e.,
\begin{equation}
\label{9}
\Phi(\Bx,\By)=\vspace{0.15cm}\left\{\begin{array}{ll}
\vspace{0.15cm}\displaystyle\frac{e^{ik|\Bx-\By|}}{4\pi|\Bx-\By|}, \mbox{ for } d=3\vspace{0.15cm}\\
\frac{i}{4}H_0^{(1)}(k|\Bx-\By|) , \mbox{ for } d=2\end{array}\right.
\end{equation}
with $H_{0}^{(1)} = J_{0} + iY_{0}$ representing the Hankel function of first type. Note that in \eqref{8} the integrals are to be understood as singular integrals defined through an operator extension from $C(\partial D_a)$. We will also consider $k$ such that
\begin{eqnarray}
\label{ass-lambda-0}
                 && \mathbf{1)} \, -k^2 \mbox{ is not a Neumann eigenvalue for the Laplace operator in $D_a$ or $B_R(\B0)$},\nonumber\\
                 && \mathbf{2)} \, -k^2 \mbox{ is not a Dirichlet eigenvalue for the Laplace operator in $D_c$}.
\end{eqnarray}
As in \cite{O2} we introduce the adjoint operator $K^{*}: \Xi \to L^{2}(\partial D_{a})$, which can be shown to satisfy
\begin{equation}
K^{*}\psi(\mathbf{x}) = \int_{\partial D_{c}} \psi_{1}(\mathbf{y}) \overline{\frac{\partial \Phi(\mathbf{y},\mathbf{x})}{\partial \nu_{\mathbf{x}}}}\,dS_{\mathbf{y}} + \int_{\partial B_{R}} \psi_{2}(\mathbf{y}) \overline{\frac{\partial \Phi(\mathbf{y},\mathbf{x})}{\partial \nu_{\mathbf{x}}}}\, dS_{\mathbf{y}}, \quad \mathbf{x} \in \partial D_{a}. \label{eq:dlpotentialadjoint}
\end{equation}

This paper proposes a sensitivity study for the following problem: Let $V\Subset D_c$ and $R'>R$. For a fixed wave number $k > 0$ and fixed $0 < \mu \ll 1$, find a function $h \in C(\partial D_{a})$ such that there exists $u \in C^{2}(\RR^{n} \setminus \overline{D_{a}}) \cap C^{1}(\RR^{n} \setminus D_{a})$ solving
\begin{equation}
\left\{ \begin{array}{rl}
(\Delta + k^{2}) u(\mathbf{x}) & = 0 \quad \mathbf{x} \in \RR^{n} \setminus \overline{D_{a}}\\
u & = h \quad \textrm{ on $\partial D_{a}$}\\
\|u - f_{1}\|_{C(\overline{V})} & = \mathcal{O}(\mu)\mbox { and }\|u\|_{C(\RR^{n} \setminus B_{R'}(\mathbf{0}))} = \mathcal{O}(\mu),
\end{array}
\right. \label{eq:inverseproblem}
\end{equation}
where $f_{1}$ is a solution of the Helmholtz equation in a neighborhood of the control region $D_{c}$. In fact, by using the operator $K$ and regularity arguments it is shown in \cite{O2} that a class of solutions for problem \eqref{eq:inverseproblem} can be obtained by considering the following problem: for a fixed wave number $k>0$ satisfying conditions \eqref{ass-lambda-0}, a given function $f=(f_1,0) \in \Xi$ and $\mu > 0$, find a density function $\phi \in C(\partial D_{a})$ such that
\begin{equation}
\|K\phi - f\|_{\Xi} \leq \mu. \label{eq:controlinequality}
\end{equation}

Problem \eqref{eq:controlinequality} is in fact a Fredholm integral equation of the first kind, and it was studied in a very general setting in \cite{O2}. There the authors proved that the bounded and compact operator $K$ is also one-to-one and has a dense (but not closed) range, thus proving the existence of a class of solutions for \eqref{eq:controlinequality} (and thus for \eqref{eq:inverseproblem}). However, the fact that $K$ is compact and that its range is not closed also implies that problem \eqref{eq:controlinequality} is ill-posed. By using regularization, one can approximate a solution to problem \eqref{eq:controlinequality} with an arbitrary level of accuracy $\mu \ll 1$. There are several methods known in the literature, but we will use in this paper the Tikhonov regularization method \cite{reg1,reg2}. This method, when applied to the operator $K:L^2(\partial D_a)\rightarrow \Xi$, proposes a solution $\phi_{\alpha}\in C(\partial D_{a})$ of the form
\begin{equation}
\label{Tikhonov}
\phi_{\alpha}=(\alpha I+K^*K)^{-1}K^*f,  \textrm{ for } 0<\alpha \ll 1,
\end{equation}
where $\alpha$ is a suitably chosen regularization parameter.

It is known that $\Vert K\phi_\alpha-f\Vert_\Xi\rightarrow 0$ as $\alpha\rightarrow 0$, (see \cite{Kirsch-Book96}, Theorem 2.16), but the optimal choice of $\alpha$ is an essential step in designing a feasible method (e.g., finding a minimal norm solution), and there are various modalities to do this. In this paper we will use the Morozov discrepancy principle associated to the following weighted residual:
\begin{equation}
\label{disc-fun}
E(\phi,h)= \displaystyle\frac{1}{\|h_{1}\|_{L^{2}(\partial D_{c})}^{2}}\| K_1\phi  - h_{1}\|_{L^{2}(\partial D_{c})}^{2} + \frac{1}{2\pi R} \|K_2\phi \|_{L^{2}(\partial B_{R})}^{2},
\end{equation}
for every given $h=(h_1,0)\in \Xi$. 
%\begin{equation}
%\label{space-Xi'}
%\Xi' = L^{2}(\partial D_{c}, \|f_{1}\|^{-2}dS) \times L^{2}(\partial B_{R}, (2\pi R)^{-1}dS).
%\end{equation}
The reasoning behind using the weighted residual discrepancy functional defined at \ref{disc-fun} is as follows. Due to the asymptotic behavior of $\frac{\partial \Phi(\mathbf{x},\, \mathbf{y})}{\partial \nu_{\mathbf{y}}} = \mathcal{O}(|\mathbf{x}-\mathbf{y}|^{-1/2})$ as $|\mathbf{x}-\mathbf{y}| \to \infty$, we have that given a fixed density $\phi$, $\|K\phi\|_{L^{2}(\partial B_{R})} = \mathcal{O}(1)$ as $R \to \infty$. In other words, using the space $L^{2}(\partial B_{R})$ with the standard surface measure is not really suited to the decay properties of double layer potential solutions, because the decay of the normal derivative $\partial_{\nu}\Phi$ is too weak. Similarly, we use the relative norm
\begin{equation}
\label{eq:F1}
\frac{\|K_1\phi - h_{1}\|_{L^{2}(\partial D_{c})}}{\|h_{1}\|_{L^{2}(\partial D_{c})}}
\end{equation}
on $\partial D_{c}$ because this is a useful quantity for determining how good the control is, regardless of the norm of $h_{1}$. Thus our procedure for finding an approximate solution for problem \eqref{eq:controlinequality} is to first make use of the Tikhonov regularization for the operator  $K:L^2(\partial D_a)\rightarrow \Xi$ as described in \eqref{Tikhonov} to obtain $\phi_\alpha$ and then apply the Morozov's discrepancy principle for the unique choice of $\alpha$ (\cite{Kress-Book99}), i.e. such that
\begin{equation}
\label{Morozov0}
E(\phi_{\alpha}, f)=\delta^2,
\end{equation}
with $\displaystyle\delta^2\leq \mu^2\min\left\{\frac{1}{2\|f_{1}\|^2_{L^{2}(\partial D_{c})}}, \frac{1}{4\pi R}\right\}$.

In what follows, we will account for noise and measurement errors and will consider \eqref{Morozov0} with $f=(f_1,0)\in \Xi$ replaced by $f_\epsilon=(f_{\epsilon,1}, f_{\epsilon,2})\in \Xi$, given by
\begin{equation}
\label{random-f}
f_{\epsilon} = (f_{1} + \epsilon s, \, 0)\in \Xi,
\end{equation}
where $s\in L^{2}(\partial D_{c})$  is a random perturbation with $\|s\|_{L^{2}(\partial D_{c})} \leq 2 \|f_{1}\|_{L^{2}(\partial D_{c})}$ and $f_1$ is a solution of the Helmholtz equation in a neighborhood of the control region $D_{c}$. We mention that in the numerical experiments of Section \ref{sec:numerics}, $f_{1}$ denotes the $k$ frequency component of the far field of a far field observer. Note that this assumption about the interrogating signal ensures that $f_1$ is a solution of the Helmholtz equation in $B_R$. In the noisy case (i.e. when $f$ is replaced by $f_\epsilon$) equation \eqref{Morozov0} becomes
\begin{equation}
\label{Morozov}
E(\phi_{\alpha}, f_\epsilon)=\delta^2,
\end{equation}
where $\phi_{\alpha}=(\alpha I+K^*K)^{-1}K^*f_\epsilon$ is the Tikhonov regularization solution. From the definition of $E$ and classical results, \cite{Kirsch-Book96,Kress-Book99}, it follows that \eqref{Morozov} admits at least a solution $\alpha$. Moreover, as we will discuss in Section \ref{sec:stability}, motivated by numerical evidence, we formulate the hypothesis that there exists $\epsilon_0>0$ such that for each $\epsilon\in (0,\epsilon_0)$, problem \eqref{Morozov} has a unique solution $\alpha(\epsilon)$ which uniquely defines a differentiable function $\epsilon\rightsquigarrow\alpha(\epsilon)$. We will study the minimal norm solution uniquely determined by \eqref{Morozov}, discuss its stability for given noisy data in $\Xi$ and, in the case of data corresponding to a point source, analyze its sensitivity with respect to parameters such as: mutual distances between $D_a$, $D_c$ and $B_R(\B0)$; wave number $k$; and the location of the point source.

\section{Stability estimate for the Tikhonov regularization}
\label{sec:stability}
In this section we present analytical and numerical arguments which indicate the stability of the minimum norm solution $\phi_{\alpha}$ with respect to noise level $\epsilon$ for a given fixed discrepancy level $\delta$. Next, we present below Lemma \ref{lemma:phif1constant} which will provide bounds for $\|f_{1}\|_{L^{2}(\partial D_{C})}$ and $\alpha$ in terms of the operatorial norm of $K_{1}^{*}$.

\begin{lemma} Let $0<\delta<\frac{1}{\sqrt{2}}$ and $z=(z_1,0) \in \Xi'$ with $z_1\neq 0$. Consider the Tikhonov regularization solution $\phi_{\alpha}=(\alpha I + K^{*}K)^{-1}K^*z \in C(\partial D_a)$, with $\alpha$ such that $\|K\phi_{\alpha} - z\|_{\Xi'} \leq \delta$. Then we have
\begin{align}
\|z_{1}\|_{L^2(\partial D_c)} & \leq 4\|K_1^*\|_{\mathcal O}\|\phi_{\alpha}\|_{L^{2}(\partial D_{a})},\label{lb-phi}\\
\alpha & \leq 4\delta \Vert K_1^*\Vert^2_{\mathcal O},\label{ub-alpha}
\end{align}
where $K_1^*$ is the adjoint operator for $K_1$ defined by \eqref{8} and $\|\cdot\|_{\mathcal O}$ denotes the operatorial norm.
 \label{lemma:phif1constant}
\end{lemma}
\begin{proof} We will start with the proof of \eqref{lb-phi}. Note that since $E(\phi_{\alpha}, z)=\delta^{2}$, we have
\begin{equation}
\label{5.17}
\|K_{1}\phi_{\alpha} - z_{1}\|_{L^{2}(\partial D_{c})}^{2} \leq \delta^{2}\|z_{1}\|_{L^2(\partial D_c)}^{2}.
\end{equation}
From \eqref{5.17} we obtain
\begin{equation}
\|K_{1}\phi_{\alpha}\|_{L^2(\partial D_c)}^{2} - 2 \, \re(K_{1}\phi_{\alpha}, \, z_{1})_{L^2(\partial D_c)} + \|z_{1}\|_{L^2(\partial D_c)}^{2} \leq \delta^{2}\|z_{1}\|_{L^2(\partial D_c)}^{2}, \end{equation}
and this implies
\begin{equation}
\|z_{1}\|_{L^2(\partial D_c)}^{2}(1 -\delta^{2}) \leq 2 \, \re(\phi_{\alpha}, \, K_{1}^{*}z_{1})_{L^2(\partial D_a)} \leq 2 \|\phi_{\alpha}\|_{L^2(\partial D_a)}\|K_{1}^{*} z_{1}\|_{L^2(\partial D_a)}. \label{pas1}
\end{equation}
%\begin{eqnarray}
%&&\!\!\!\!\!\!\!\!\|K_{1}\phi_{\alpha}\|_{L^2(\partial D_c)}^{2} - 2 \, \re(K_{1}\phi_{\alpha}, \, z_{1})_{L^2(\partial D_c)} + \|z_{1}\|_{L^2(\partial D_c)}^{2} \leq \delta^{2}\|z_{1}\|_{L^2(\partial D_c)}^{2}\nonumber\\
%&&\label{pas1}\\
%&&\!\!\!\!\!\!\!\!\!\!\!\!\!\!\!\|z_{1}\|_{L^2(\partial D_c)}^{2}(1 -\delta^{2}) \leq 2 \, \re(\phi_{\alpha}, \, K_{1}^{*}z_{1})_{L^2(\partial D_a)}\leq 2 \|\phi_{\alpha}\|_{L^2(\partial D_a)}\|K_{1}^{*} z_{1}\|_{L^2(\partial D_a)}\nonumber.
%\end{eqnarray}
Then \eqref{pas1} gives
\begin{align}
\|\phi_{\alpha}\|_{L^2(\partial D_a)} & \geq \frac{\|z_{1}\|_{L^2(\partial D_c)}^{2}(1 - \delta^{2})}{2 \|K_{1}^{*}z_{1}\|_{L^2(\partial D_a)}}
= \paren{\frac{1 - \delta^{2}}{2}}\paren{ \frac{\|z_{1}\|_{L^2(\partial D_c)}}{\|K_{1}^{*}z_{1}\|_{L^2(\partial D_a)}} }\|z_{1}\|_{L^2(\partial D_c)}\notag\\
& \geq \frac{1 - \delta^{2}}{2} \paren{\frac{\|z_{1}\|_{L^2(\partial D_c)}}{\|K_{1}^{*}\|_{\mathcal O}}}\notag\\
&\geq \frac{\|z_1\|_{L^2(\partial D_c)}}{4\|K_{1}^{*}\|_{\mathcal O}}.\label{margine}
\end{align}

Next we proceed towards proving \eqref{ub-alpha}. From the definition of $\phi_{\alpha}$ we have
\begin{eqnarray}
\alpha \phi_{\alpha} + K^{*}K\phi_{\alpha}&=& K^*z=K_1^*z_1,\nonumber\\
\alpha \phi_{\alpha}+K_2^{*}K_2\phi_{\alpha} &=& K_1^*(z_1-K_1\phi_{\alpha}).\label{5.18}
\end{eqnarray}
Here we have used from \eqref{8} and \eqref{eq:dlpotentialadjoint} that
\begin{eqnarray}
\label{adjunct-rel}
K^*\psi=K_1^*\psi_1+K^*_2\psi_2, \mbox{ for all }\psi\in\Xi,\nonumber\\
K^*Kv=K_1^*K_1v+K_2^*K_2v, \mbox{ for all }v\in L^2(\partial D_a).
\end{eqnarray}
Multiplying \eqref{5.18} with $\phi_{\alpha}$ in the sense of the usual scalar product in
$L^2(\partial D_a)$, we obtain
\begin{equation}
\label{rel11}
\alpha \|\phi_{\alpha}\|^2_{L^2(\partial D_a)} +\|K_2\phi_{\alpha}\|^2_{L^2(\partial D_a)}  = (K_1^*(z_1-K_1\phi_{\alpha}),\phi_{\alpha})_{L^2(\partial D_a)}.\end{equation}
Using \eqref{5.17}, \eqref{margine} and \eqref{adjunct-rel} in \eqref{rel11} we then have
\begin{equation*}\alpha \|\phi_{\alpha}\|_{L^2(\partial D_a)} \leq \delta\|K_1^*\|_{\mathcal O}\|z_1\|_{L^2(\partial D_c)}\Longrightarrow
\alpha\leq 4\delta\|K_1^*\|^2_{\mathcal O} \label{alfa}.
\end{equation*}
\end{proof}

Next, before presenting the main stability result of this section, i.e., Proposition \ref{theorem:stabilityestimate} below, we must understand the conditions on $\epsilon > 0$ under which \eqref{Morozov} admits a unique solution $\alpha(\epsilon)$ with the property that the resulting function $\epsilon\rightsquigarrow \alpha(\epsilon)$ is differentiable. For this, we consider the function $g:(0,\infty)\times(0,\infty)\rightarrow (0,\infty)$ defined by
\begin{equation}
g(\alpha, \epsilon) = E(\phi_\alpha, f_\epsilon)\label{eq:galphaeps},
\end{equation}
where $f_{\epsilon} \in \Xi$ was introduced in \eqref{random-f}, and $\phi_\alpha$ is the Tikhonov regularization solution introduced in \eqref{Morozov}. With this notation, \eqref{Morozov} can be rewritten as
\begin{equation}
\label{Morozov-g}
g(\alpha, \epsilon) = \delta^2,
\end{equation}
where $\delta$ is the desired fixed discrepancy level. By using classical results (e.g., \cite{Kirsch-Book96,Kress-Book99}) it can be observed that for every $\epsilon$, \eqref{Morozov-g} admits at least one solution in $(0,\infty)$ and that $g$ defined by \eqref{eq:galphaeps} is differentiable with respect to positive $\alpha$ and $\epsilon$, respectively. In fact, it follows from classical arguments that a maximum value of $\alpha$ for a given $\epsilon$ exists. This solution of \eqref{Morozov-g} corresponds to the $L^2$ minimal energy solution and we will further refer to it as the Morozov solution.

For the remainder of the paper, unless otherwise specified, $C$ will denote a generic constant which depends only on the operator $K$, $d_c=diam(D_c)$ and $d=dist(\partial D_c,\partial D_a)$. The next Proposition states a central stability result concerning the Morozov solution of \eqref{Morozov-g}. We have,

\begin{prop}
Let $0<\delta$ be as above, and $f_\epsilon$ and $f_1$ as defined in \eqref{random-f}. For every $\epsilon \geq 0$ consider $\phi_{\alpha_{\epsilon}}=(\alpha_\epsilon I + K^{*}K)^{-1}K^{*}f_{\epsilon}\in C(\partial D_a)$ with $\alpha_\epsilon=\alpha(\epsilon)$ the Morozov solution of \eqref{Morozov-g}. Then we have,
\begin{equation}
\displaystyle\frac{\| \phi_{\alpha_\epsilon} - \phi_{\alpha_0}\|_{L^{2}(\partial D_{a})}}{\|\phi_{\alpha_\Ge}\|_{L^{2}(\partial D_{a})}} \leq \displaystyle\frac{ \left|\displaystyle\frac{\alpha_\Ge}{\alpha_0} - 1\right| + \sqrt{ \left|\displaystyle\frac{\alpha_\Ge}{\alpha_0} -1 \right|^{2} + 16\displaystyle \frac{\epsilon\paren{ 2\delta + C\delta\epsilon +C\epsilon}}{\alpha_0}\|K^*_{1}\|_{\mathcal O}}}{2}.
\end{equation}\label{theorem:stabilityestimate}
\end{prop}
\begin{proof}Fix $\epsilon > 0$ and let $f=f_\epsilon\mbox{ for } \epsilon=0$. Let us recall that $\alpha(\epsilon)$ is uniquely implicitly defined by the equation $E((\alpha_\Ge I+K^*K)^{-1}K^*f_\Ge, f_{\epsilon})= \delta^2$ and by Lemma \ref{lemma-diff-alpha} it will be differentiable in some interval $(0,\epsilon_0)$ for all wavenumbers $k$.  Next consider
\begin{align*}
\alpha_\Ge \phi_{\alpha_\Ge} + K^{*}K\phi_{\alpha_\Ge} & = K^{*}f_{\epsilon},\\
\alpha_{0} \phi_{\alpha_{0}} + K^{*}K\phi_{\alpha_0} & = K^{*}f.
\end{align*}
Subtracting, we obtain
\begin{align}
 \alpha_{0}\phi_{\alpha_0}-\alpha_\Ge \phi_{\alpha_\Ge} + K^{*}K(\phi_{\alpha_0}-\phi_{\alpha_\Ge}) & = K^{*}(f - f_{\epsilon}),\notag\\
\alpha_0(\phi_{\alpha_0} - \phi_{\alpha_{\Ge}}) + (\alpha_0 - \alpha_\Ge)\phi_{\alpha_{\Ge}} + K^{*}K(\phi_{\alpha_0}-\phi_{\alpha_\Ge}) & = K^{*}(f - f_{\epsilon}).\label{3.1.1}
\end{align}
Integrating both sides of \eqref{3.1.1} against $\phi_{\alpha_0} - \phi_{\alpha_{\Ge}}$ yields
\begin{align}
& \displaystyle\alpha_0 \|\phi_{\alpha_0} - \phi_{\alpha_{\Ge}}\|_{L^2(\partial D_a)}^{2}  + (\alpha_0 - \alpha_{\Ge})(\phi_{\alpha_{\Ge}}, \, \phi_{\alpha_0} - \phi_{\alpha_{\Ge}})_{L^2(\partial D_a)} + \|K(\phi_{\alpha_0} - \phi_{\alpha_{\Ge}})\|_\Xi^{2}\notag\\
& \quad\quad= \displaystyle(K(\phi_{\alpha_0} - \phi_{\alpha_{\Ge}}), \, f - f_{\epsilon})_{\Xi}= \displaystyle(K_{1}(\phi_{\alpha_0} - \phi_{\alpha_{\Ge}}), \, f_{1} - f_{\epsilon, 1})_{L^2(\partial D_c)}.\label{3.1.2}
\end{align}
where, we have used \eqref{8} and the fact that $f_2=f_{\epsilon,2}=0$ in the last equality above. Thus,
\begin{eqnarray}
\label{ecuatie-1}
\!\!\!\!\!\!\!\!\!\alpha_0 \|\phi_{\alpha_0} - \phi_{\alpha_{\Ge}}\|_{L^2(\partial D_a)}^{2}  \leq & |\alpha_0 - \alpha_{\Ge}|\|\phi_{\alpha_{\Ge}}\|_{L^2(\partial D_a)}\|\phi_{\alpha_0} - \phi_{\alpha_{\Ge}}\|_{L^2(\partial D_a)} \nonumber\\
+ & \|f_{1} - f_{\epsilon, 1}\|_{L^2(\partial D_c)}\|K_{1}(\phi_{\alpha_0} - \phi_{\alpha_{\Ge}})\|_{L^2(\partial D_c)}
\end{eqnarray}
Observe that
\begin{align}
\|K_{1}(\phi_{\alpha_0} - \phi_{\alpha_{\Ge}})\|_{L^2(\partial D_c)} & \leq \|K_{1} \phi_{\alpha_0} - f_{ 1}\|_{L^2(\partial D_c)} + \|K_{1} \phi_{\alpha_{\Ge}} - f_{\Ge,1}\|_{L^2(\partial D_c)} + \|f_{1} - f_{\epsilon, 1}\|_{L^2(\partial D_c)}\nonumber\\
&\leq \Gd \|f_{1}\|_{L^2(\partial D_c)}+ \Gd \|f_{\Ge,1}\|_{L^2(\partial D_c)} + C\epsilon \|f_{1}\|_{L^2(\partial D_c)}\nonumber\\
& = \paren{ 2\delta + C\delta\epsilon +C\epsilon}\|f_{1}\|_{L^2(\partial D_c)}. \label{ecuatie-2}
\end{align}
where $f_{\epsilon,1}=f_{1} + \epsilon s$ with $\|s\|_{L^{2}(\partial D_{c})} \leq C \|f_{1}\|_{L^{2}(\partial D_{c})}$, and we have used the definition of $\phi_{\alpha_{\epsilon}}$ and \eqref{disc-fun} in the inequalities above. Using \eqref{ecuatie-2} in \eqref{ecuatie-1} we obtain
\begin{align}
\alpha_0 \|\phi_{\alpha_0} - \phi_{\alpha_{\Ge}}\|_{L^2(\partial D_a)}^{2} & \leq |\alpha_0 - \alpha_{\Ge}| \|\phi_{\alpha_{\Ge}}\|_{L^2(\partial D_a)}\|\phi_{\alpha_0} - \phi_{\alpha_{\Ge}}\|_{L^2(\partial D_a)} \notag\\
& + \epsilon\paren{ 2\delta + C\delta\epsilon +C\epsilon}\|f_{1}\|^2_{L^2(\partial D_c)}. \label{ecuatie-3}
\end{align}
If we define $A := \frac{\|\phi_{\alpha_\Ge} - \phi_{\alpha_{0}} \|_{L^{2}(\partial D_{a})}}{\|\phi_{\alpha_{\Ge}}\|_{L^{2}(\partial D_{a})}}$, inequality \eqref{ecuatie-3} implies that
\begin{align}
\alpha_0 A^{2} & \leq |\alpha_0 - \alpha_{\Ge}| A + \frac{\epsilon \paren{ 2\delta + C\delta\epsilon +C\epsilon} \|f_{1}\|_{L^{2}(\partial D_{c})}^{2}}{\|\phi_{\alpha_{\Ge}}\|_{L^{2}(\partial D_{a})}^{2}} \notag\\
& \leq |\alpha_0 - \alpha_{\Ge}| A + 16\epsilon \paren{ 2\delta + C\delta\epsilon +C\epsilon}\|K^*_1\|^2_{\mathcal O}, \label{quadratic}
\end{align}
where we have used \eqref{lb-phi} of Lemma \ref{lemma:phif1constant} in the last inequality above. Next, consider $$h(A) := A^{2} - \left|\displaystyle\frac{\alpha_{\Ge}}{\alpha_0} - 1 \right| A - 16\|K^*_1\|^2_{\mathcal O}\displaystyle\frac{\epsilon \paren{ 2\delta + C\delta\epsilon +C\epsilon}}{\alpha_0}.$$ Then, from \eqref{quadratic} we have $h(A) \leq 0$ and this implies that
\begin{equation*}
A \leq \displaystyle\frac{ \left|\displaystyle\frac{\alpha_{\Ge}}{\alpha_0} - 1 \right| + \sqrt{\left|\displaystyle\frac{\alpha_{\Ge}}{\alpha_0} - 1 \right|^{2} + 64\|K^*_1\|^2_{\mathcal O}\displaystyle\frac{\epsilon \paren{ 2\delta + C\delta\epsilon +C\epsilon}}{\alpha_0}}}{2}
\end{equation*}
which completes the proof.\end{proof}

Regarding the monotonic character of $g$,  we note that, as suggested by the numerics, $g$ is not in general \textit{globally} monotonic with respect to $\alpha$ as can be seen in Figure \ref{fig:nonincreasing_F}, which considers an antenna of radius $a = 0.01$, region of control characterized in polar coordinates by $r_{1} = 0.011$, $r_{2} = 0.015$, $\theta \in [3\pi/4, 5\pi/4]$, wave number $k=10$, $f_\epsilon$ given by \eqref{random-f} with $f_1=\frac{1}{4}H^{(1)}_0(k|\Bx-\Bx_0|)$ with $\mathbf{x}_{0} = [10000,0]^{T}$, and noise level $\epsilon=0.005$. But on the other hand, for the same geometry and functional settings as in Figure \ref{fig:nonincreasing_F}, we observe in Figure \ref{fig:plot_g_dg_alpha} that for each $\epsilon<0.015$, $g(\alpha, \epsilon) = E(\phi_{\alpha_{\epsilon}}, \, f_{\epsilon})$ is strictly increasing with respect to $\alpha$ in the interval $(10^{-4}, 1)$. Moreover, for every $\epsilon<0.015$ the Morozov solution $\alpha(\epsilon)$ is the unique solution of \eqref{Morozov-g} in $(10^{-4}, 1)$.

\begin{figure}
\includegraphics[trim = 20 20 20 20, clip = true, width=\textwidth]{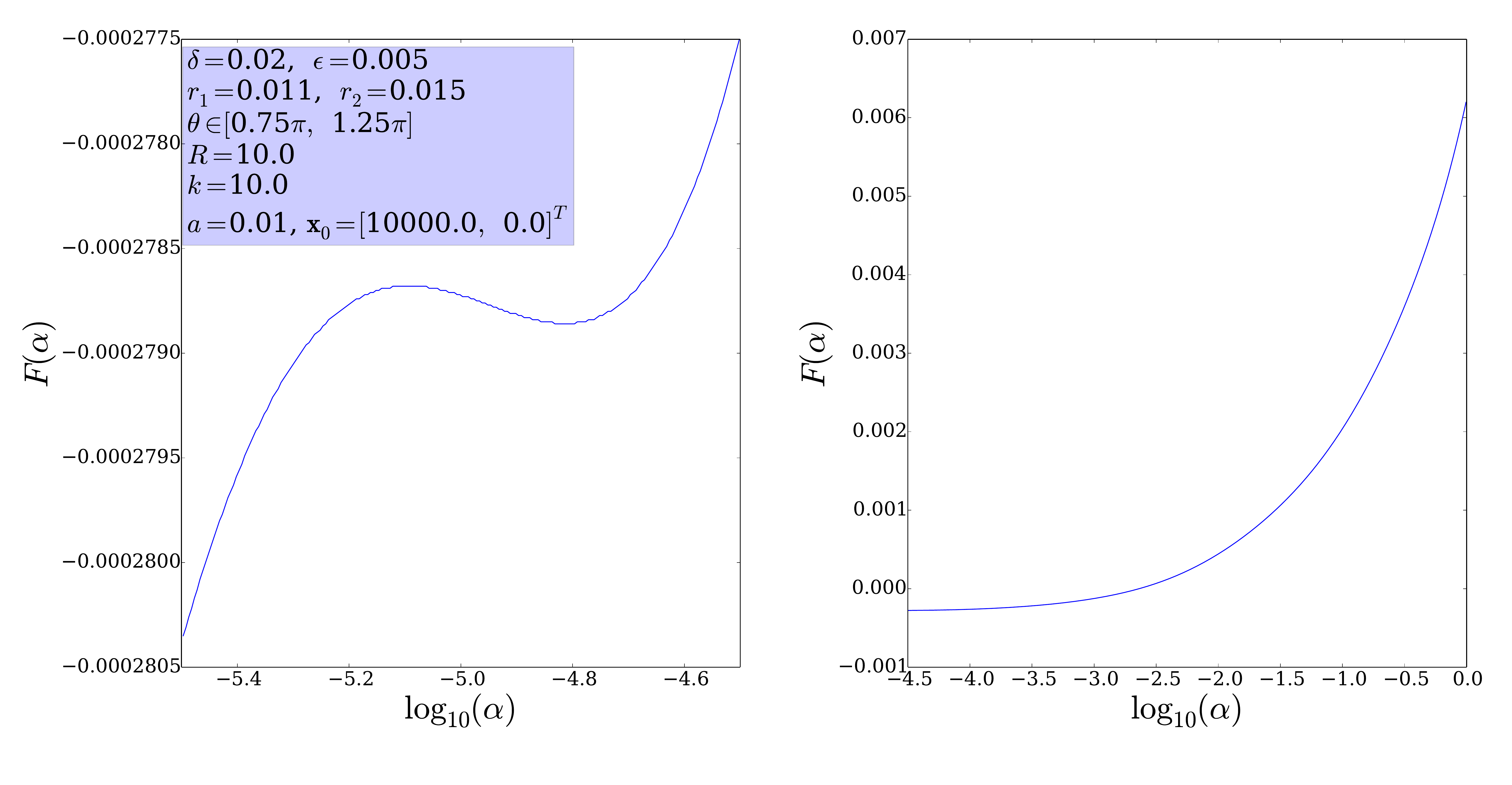}
\caption{Plot of $F(\alpha) = E(\phi_\alpha, f_\epsilon) - \delta^{2}$ for $\epsilon=0.005$ and $\alpha$ in a range where one can see its non-monotonic behavior below a particular threshold.}\label{fig:nonincreasing_F}
\end{figure}

In fact, for the same geometry and functional settings as in Figure \ref{fig:plot_g_dg_alpha} and for $k \in [1,100]$ and $\epsilon = 0.005$, Table \ref{tab:pktable} summarizes the values of $p_k>0$ for which $g(\alpha, \epsilon)$ is locally strictly monotonic with respect to $\alpha$ in an interval $(10^{-p_k}, 1)$, as well as the value of the Morozov solution for each $k$ (also see Figure \ref{fig:pk_plot}). Together with Figure \ref{fig:plot_g_dg_alpha} where $k=10$ and $\epsilon$ is varied in the interval $[0,0.015]$, this suggests that the Morozov solution $\alpha(\epsilon)$ of \eqref{Morozov-g} satisfies $\alpha(\epsilon)\in(10^{-p_k}, 1)$ at least for $\epsilon < 0.015$. This in turn together with the strict monotonicity $g$ implies the existence of a unique solution $\alpha(\epsilon)$ for \eqref{Morozov-g}. Then, uniqueness together with the fact that $\frac{\partial}{\partial \alpha}g(\alpha, \epsilon)\neq 0$ in $ (10^{-p_k}, 1)\times (0,0.015)$ (for $k=10$ shown in  Figure \ref{fig:plot_g_dg_alpha}) implies the differentiability of the Morozov solution $\alpha(\epsilon)$ by using the implicit function theorem.

\begin{figure}
\includegraphics[trim = 20 20 20 20, clip = true, width=\textwidth]{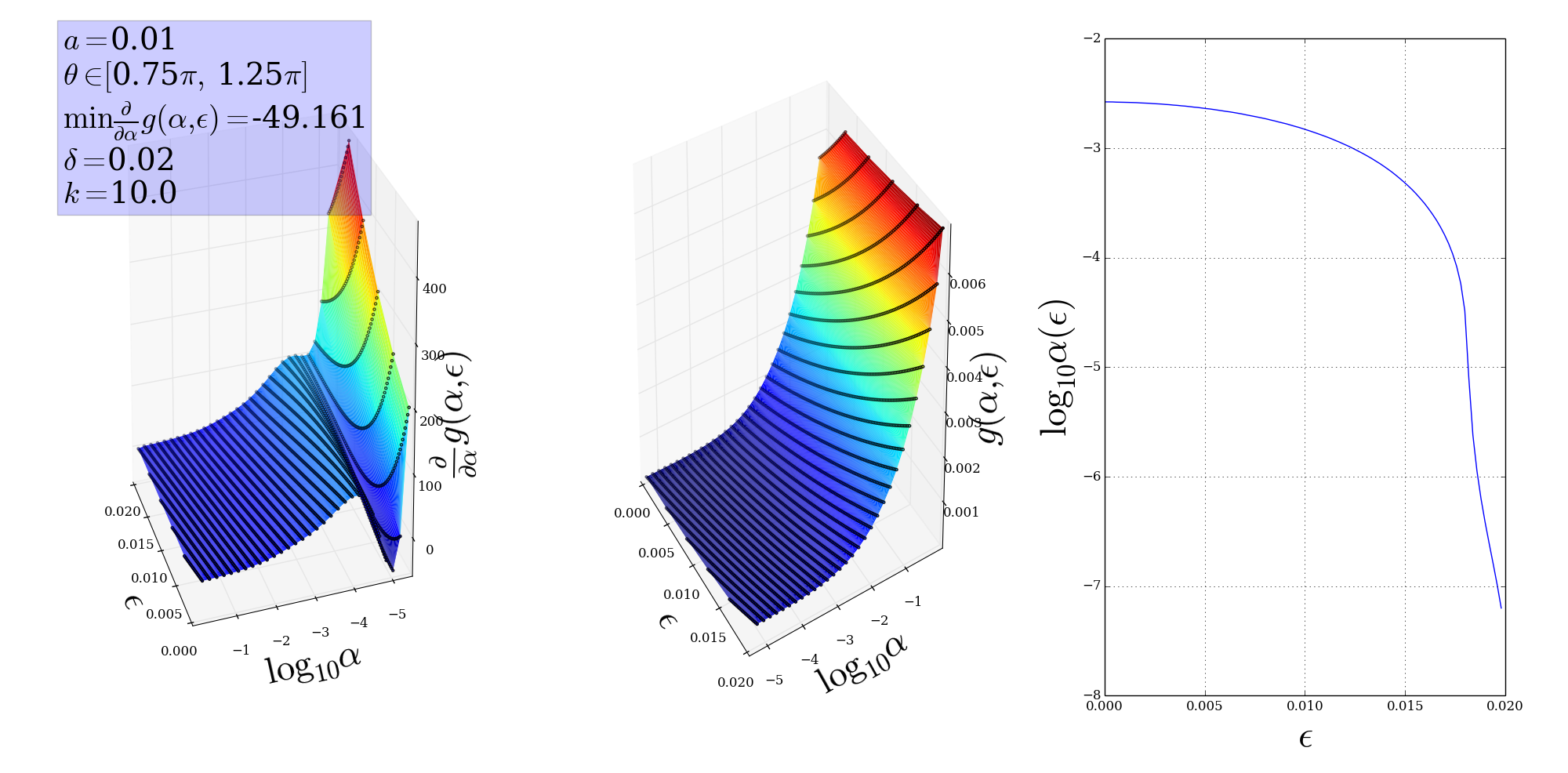}
\caption{Plot of $g(\alpha, \epsilon) = E(\phi_{\alpha_\epsilon}, f_\epsilon)$ and $\frac{\partial}{\partial \alpha}g(\alpha, \epsilon)$ with respect to $\alpha, \epsilon$ together with the unique largest value $\alpha(\epsilon)$ such that $E(\phi_{\alpha_\epsilon}, f_\epsilon) = \delta^{2}$, where $\delta = 0.02$ is fixed.}\label{fig:plot_g_dg_alpha}
\end{figure}

\begin{table}
\begin{center} 
\begin{tabular}{||l l l| l l l||} 
\hline $k$ & $-p_{k}$ & Morozov $\alpha$ & $k$ & $-p_{k}$ & Morozov $\alpha$ \\ [0.5ex] 
\hline \hline 
1.0 & -5.74057337341 & 0.0021397 & 51.0 & -5.63387601498 & 0.023987\\ 
\hline 
6.0 & -5.15371857022 & 0.0022445 & 56.0 & -5.5538513243 & 0.028226\\ 
\hline 
11.0 & -4.61487654411 & 0.0027985 & 61.0 & -5.58052339557 & 0.032734\\ 
 \hline 
16.0 & -4.43348001737 & 0.0037643 & 66.0 & -7.93866063316 & 0.038053\\ 
 \hline 
21.0 & -7.22374205154 & 0.0052228 & 71.0 & -7.78393971408 & 0.043891\\ 
 \hline 
26.0 & -6.73292218326 & 0.0072707 & 76.0 & -7.60786606025 & 0.048884\\ 
 \hline 
31.0 & -6.41280555768 & 0.0098052 & 81.0 & -7.41580451387 & 0.05425\\ 
 \hline 
36.0 & -6.18339417827 & 0.012607 & 86.0 & -7.2077344143 & 0.060871\\ 
 \hline 
41.0 & -5.97530913764 & 0.015782 & 91.0 & -7.00500135956 & 0.066982\\ 
 \hline 
46.0 & -5.78858575492 & 0.01971 & 96.0 & -6.90364624001 & 0.07278\\ 
 \hline 
\end{tabular} 
\end{center}
\caption{Table of values $-p_{k}$ such that $g(\alpha, \epsilon)$ is increasing with respect to $\alpha$ for $\alpha \geq 10^{-p_{k}}$. $\epsilon$ is fixed at $0.005$ in this case.}\label{tab:pktable}
\end{table}

\begin{figure}
\vspace{-3cm}
\includegraphics[width=\textwidth]{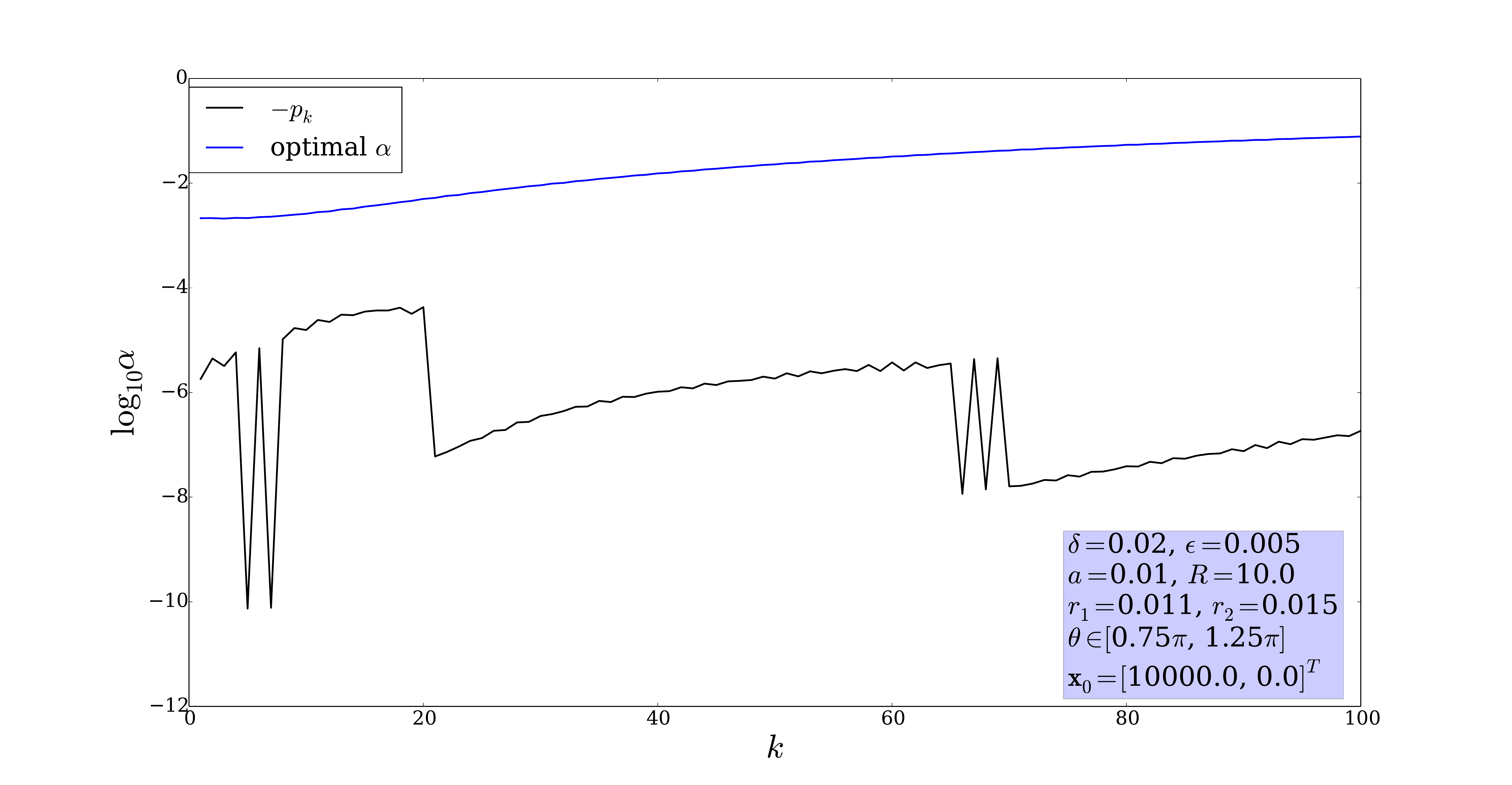}
\caption{Threshold value $p_{k}$ for which $F(\alpha)$ is increasing when $\alpha > 10^{-p_{k}}$. Also shows the value of $\alpha$ for which $F(\alpha) = \delta^{2}$ with the same setting as in Figure \ref{fig:plot_g_dg_alpha}, where $\delta = 0.02$ and the noise level $\epsilon = 0.005$.}\label{fig:pk_plot}
\end{figure}

For simplicity of notation, in what follows we will write sometimes $\alpha$ instead of $\alpha_\epsilon=\alpha(\epsilon)$ and we will use $\alpha'$ and $f_{\epsilon,i}'$ to denote $\frac{d\alpha}{d\epsilon}$ and $\frac{df_{\epsilon,i}}{d\epsilon}$ respectively. Motivated by the above numerics, we formulate the following more general hypothesis:

\begin{hypothesis}
\label{hypothesis-I}
Assume the same geometrical setup as in Section \ref{sec:background} and let $f_\epsilon, f_1$ be as in \eqref{random-f}. Then there exists $p_0>0$ and $\epsilon_0>0$ such that $\frac{\partial}{\partial \alpha}g(\alpha, \epsilon)\neq 0$ in $ (10^{-p_0}, 1)\times (0,\epsilon_0)$ for all wave numbers $k$, and the Morozov solution $\alpha(\epsilon)$ is the unique solution of \eqref{Morozov-g} in $(10^{-p_0}, 1)$.
\end{hypothesis}
For example, as shown in Table \ref{tab:pktable} and Figure \ref{fig:pk_plot}, for $f_\epsilon, f_1$ as in \eqref{random-f} with $s=\widehat{\nu}\Vert f_1\Vert_{ L^2(\partial D_c)}$ where $\widehat{\nu}\in L^2(\partial D_c)$ is a random perturbation with $\Vert \widehat{\nu}\Vert_{ L^2(\partial D_c)}=1$, and for the same geometry and data as in Figure  \ref{fig:plot_g_dg_alpha}, we have that Hypothesis \ref{hypothesis-I} is satisfied for $p_0=10^{-4}$ and $\epsilon_0=0.015$ for all $k=\overline{1,100}$. Thus, whenever Hypothesis \ref{hypothesis-I} is satisfied, the definition of $\alpha(\epsilon)$ and the implicit function theorem imply:

\begin{lemma}
\label{lemma-diff-alpha}
There exists $\epsilon_0>0$ such that for every $\epsilon\in(0,\epsilon_0)$, the function $\alpha:(0,\epsilon_0)\rightarrow(0,\infty)$, where for each $\epsilon\in(0,\epsilon_0)$, $\alpha(\epsilon)$ represents the Morozov solution of \eqref{Morozov-g},  will be differentiable for all wave numbers $k$.
\end{lemma}

The next Lemma is a technical result needed in the stability estimate obtained in Corollary \ref{corolar-1}.
\begin{lemma}
\label{teorema-int}
Let $f_{1}$ be a solution of the Helmholtz equation in a neighborhood of $D_c$ satisfying the following source type condition:
\begin{equation}
\label{source-cond}
\displaystyle\left\Vert K_1\psi_0-\frac{f_1}{\Vert f_1\Vert_{L^{2}(\partial D_{c})}}\right\Vert_{L^{2}(\partial D_{c})}\!\!\!\!\leq C\delta \mbox{ for some }\psi_0\in L^2(\partial D_a) \mbox{ with } \Vert\psi_0\Vert_{L^2(\partial D_a)}\leq C \delta.
\end{equation}
Assume $R$ (radius of $B_R(\B0)$) is such that,
\begin{equation}
\label{R-bound}
\Vert f_1\Vert_{L^{2}(\partial D_{c})} \leq \sqrt{\pi R}.
\end{equation}
Consider $s_\epsilon=f_{1} + \frac{\epsilon}{2} \widehat{\nu}\|f_{1}\|_{L^{2}(\partial D_{c})}$ where $\widehat{\nu}\in L^2(\partial D_c)$ is a random perturbation with $\|\widehat{\nu}\| \leq 1$. Assume the same functional framework as in Proposition \ref{theorem:stabilityestimate} and that Hypothesis \ref{hypothesis-I}  holds true in the case when $f_\epsilon$ is given by
\begin{equation}
\label{hyp-1}
f_{\epsilon} = (f_{\epsilon,1}, f_{\epsilon,2})= (f_{1} + \epsilon s_\epsilon, \, 0).
\end{equation}
Then, there exists $\epsilon_0>0$ such that the Morozov solution of equation \eqref{Morozov-g} $\alpha=\alpha(\epsilon)$ satisfies
\begin{equation}
\alpha|\alpha'| \leq C\frac{\delta^2}{\sqrt{\alpha}}, \mbox{ for all } \epsilon<\epsilon_0. 
\end{equation}
\end{lemma}

\begin{proof} Define the weights
\begin{align*}
w_{1} & := \frac{1}{\|f_{\epsilon, 1}\|_{L^{2}(\partial D_{c})}^{2}}\\
w_{2} & := \frac{1}{2 \pi R}.
\end{align*}
and denote $T_{\alpha} := (K^{*}K + \alpha I)^{-1}$, $R_{\alpha} := T_{\alpha}K^{*}$. Then using the Einstein summation convention, we may write
\begin{equation*}
E(\phi_{\alpha}, f_{\epsilon})= w_{i}\|K_{i}\phi_{\alpha} - f_{\epsilon,i} \|_{L^{2}(W_{i})}^{2} = w^{i} \paren{K_{i}\phi_{\alpha} - f_{\epsilon,i}, \, K_{i}\phi_{\alpha} - f_{\epsilon,i}}_{L^{2}(W_{i})},
\end{equation*}
where $W_{1} = \partial D_{c}$, $W_{2} = \partial B_{R}$. Next, as in Lemma \ref{lemma-diff-alpha}, we observe that Hypothesis \ref{hypothesis-I} together with the implicit function theorem imply the uniqueness and differentiability of $\alpha(\epsilon)$, on the interval $\epsilon \in (0, \epsilon_0)$ for some $\epsilon_0>0$, where $\alpha(\epsilon)$ is uniquely and implicitly defined by the equation $E((\alpha_\Ge I+K^*K)^{-1}K^*f_\Ge, \, f_{\epsilon})= \delta^2$. Differentiating the equation $E((\alpha_\Ge I+K^*K)^{-1}K^*f_\Ge, \, f_{\epsilon})= \delta^2$ with respect to $\epsilon$ and noting that $\delta$ is fixed, we obtain
\begin{align}
0  = \partial_{\epsilon}E(\phi_{\alpha}, f_{\epsilon})& = 2w_{i}\,\re{ \paren{K_{i}\phi_{\alpha}' - f_{\epsilon,i}', \, K_{i}\phi_{\alpha} - f_{\epsilon, i}}}_{L^{2}(W_{i})} \notag\\
&  \quad - 2(w_{1})^{2}\,\re{\paren{f_{\epsilon,1}', f_{\epsilon, 1}}}_{L^{2}(\partial D_{c})}\|K_{1}\phi_{\alpha} - f_{\epsilon,1}\|_{L^{2}(\partial D_{c})}^{2}\label{eq:diffnorm1}.
\end{align}

Next, from $(K^{*}K + \alpha I)\phi_{\alpha} = K^{*}f_{\epsilon}$  we observe that
\begin{equation}
\phi_{\alpha}' = R_{\alpha}f_{\epsilon}' - \alpha' T_{\alpha}\phi_{\alpha}. \label{eq:phialphadiff}
\end{equation}
Thus, we may write
\begin{equation}
K_{i}\phi_{\alpha}' - f_{\epsilon, i}' = -\alpha' K_{i} T_{\alpha}\phi_{\alpha} + K_{i}R_{\alpha}f_{\epsilon}' - f_{\epsilon, i}'.\label{eq:sub1}
\end{equation}
By using (\ref{eq:sub1}) and \eqref{eq:phialphadiff} we obtain that
\begin{align}
\label{ec-11}
\!\!\!\!\!\!\!\!\!
2\paren{K_{i}\phi_{\alpha}' - f_{\epsilon,i}', \, K_{i}\phi_{\alpha}}_{L^{2}(W_{i})} &= -2\alpha' \paren{T_{\alpha}\phi_{\alpha}, \, K_{i}^{*}K_{i}\phi_{\alpha}}_{L^{2}(\partial D_{a})}\notag\\
& + 2\paren{K_{i}R_{\alpha}f_{\epsilon}' - f_{\epsilon, i}', \, K_{i}\phi_{\alpha}}_{L^{2}(W_{i})},
\end{align}
and
\begin{align}
\label{ec-12}
\!\!\!\!\!\!\!\!\!-2\paren{ f_{\epsilon,i}, \, K_{i}\phi_{\alpha}' - f_{\epsilon, i}'}_{L^{2}(W_{i})} & = 2\alpha'\paren{ f_{\epsilon,i}, \, K_{i}T_{\alpha}\phi_{\alpha}}_{L^{2}(W_{i})} - 2\paren{f_{\epsilon,i}, \, K_{i}R_{\alpha}f_{\epsilon}' - f_{\epsilon, i}'}_{L^{2}(W_{i})}\notag\\
& = 2\alpha'\paren{ K_{i}^{*}f_{\epsilon,i}, \, T_{\alpha}\phi_{\alpha}}_{L^{2}(\partial D_{a})} - 2\paren{f_{\epsilon,i}, \, K_{i}R_{\alpha}f_{\epsilon}' - f_{\epsilon, i}'}_{L^{2}(W_{i})}.
\end{align}

Let $P,Q$ be defined by
\begin{align}
P & = 2w_{i}\left[ \paren{K_{i}\phi_{\alpha} - f_{\epsilon, i}, \, K_{i}R_{\alpha}f_{\epsilon}' - f_{\epsilon,i}'}_{L^{2}(W_{i})} + \alpha' \paren{K_{i}^{*}f_{\epsilon,i} - K_{i}^{*}K_{i}\phi_{\alpha}, \, T_{\alpha}\phi_{\alpha}}_{L^{2}(\partial D_{a})} \right]\label{P}\\
Q & = 2(w_{1})^{2}\,\paren{f_{\epsilon,1}', f_{\epsilon, 1}}_{L^{2}(\partial D_{c})}\|K_{1}\phi_{\alpha} - f_{\epsilon,1}\|^{2}.\label{Q}
\end{align}
Then from \eqref{ec-11}, \eqref{ec-12} used in \eqref{eq:diffnorm1} we obtain
\begin{equation}
\label{ec-15'}
0=\partial_{\epsilon}E(\phi_{\alpha}, f_{\epsilon})= \re(P) - \re(Q).
\end{equation}
We focus first on $P$ introduced in \eqref{P}. In this regard, let us define
\begin{equation}
\label{P1}
L_{i} := \paren{ K_{i}\phi_{\alpha} - f_{\epsilon, i}, \, K_{i}R_{\alpha}f_{\epsilon}' - f_{\epsilon, i}'}_{L^{2}(W_{i})}.
\end{equation}
Observe that \eqref{hyp-1} implies
\begin{equation}
f_\epsilon'=(f'_{\epsilon,1}, f'_{\epsilon,2})=(f_1+\epsilon\widehat{\nu}\Vert f_1\Vert_{L^2(\partial D_c)},0)\label{P2}.
\end{equation}
First note that by using classical arguments based on the singular value decomposition for $K:L^2(\partial D_a)\rightarrow\Xi$, one can adapt the results in \cite{Kirsch-Book96} (Theorem 2.7) and obtain,
\begin{equation}
\label{regl-1}
\|KR_{\alpha}Kz - Kz \|_{\Xi} \leq C\sqrt{\alpha} \Vert z\Vert_{L^2(\partial D_a)}, \mbox{ for every } z\in L^2(\partial D_a).
\end{equation}

Let $\displaystyle f=\left(\frac{f_1}{\Vert f_1\Vert_{L^2(\partial D_c)}},0\right)$ and $\displaystyle v=\frac{f_\epsilon'}{\Vert f_1\Vert_{L^2(\partial D_c)}}-f$. By using the definition of $E$ and $\Xi$, \eqref{disc-fun}, \eqref{source-cond}, \eqref{R-bound}, \eqref{P2}, \eqref{regl-1}, Cauchy's inequality and the triangle inequality in \eqref{P1}, we obtain 
\begin{eqnarray}
\label{P4}
2|w_iL_i|  &\leq& C\delta (\sqrt{w_1}\|K_1R_{\alpha}f_{\epsilon}' - f_{\epsilon,1}' \|_{L^2(\partial D_c)}+\sqrt{w_2}\|K_2R_{\alpha}f_{\epsilon}' - f_{\epsilon,2}' \|_{L^2(\partial B_R(\B0))})\nonumber\\
&\leq&C\delta (\|K_1R_{\alpha}(v+f) - (v_1+f_1) \|_{L^2(\partial D_c)}+\|K_2R_{\alpha}(v+f) - (v_2+f_2) \|_{L^2(\partial B_R(\B0))})\nonumber\\
&\leq& C\delta \|KR_{\alpha}(v+f) - (v+f) \|_{\Xi}\nonumber\\
&\leq& C\delta ( \|KR_{\alpha}f - f\|_{\Xi}\!+\!\|KR_{\alpha}v\! - \!v\|_{\Xi})\nonumber\\
&\leq& C\delta(\|KR_{\alpha}K\psi_0 - K\psi_0 \|_{\Xi}+\|KR_{\alpha}(K\psi_0 - f) \|_{\Xi}\!+\!\|K\psi_0 - f\|_{\Xi})+C\frac{\delta\epsilon}{\sqrt{\alpha}}\nonumber\\
&\leq& C\delta^2\sqrt{\alpha} +C\frac{\delta^2}{\sqrt{\alpha}}+C\delta^2+C\frac{\delta\epsilon}{\sqrt{\alpha}}\nonumber\\
&\leq& C\frac{\delta^2}{\sqrt{\alpha}},
\end{eqnarray}
where Einstein summation convention was used and where, in the second inequality above we make use of \eqref{R-bound} to obtain $\sqrt{\frac{w_2}{w_1}}\leq 1$ and $\sqrt{w_1}\Vert f_1\Vert_{L^2(\partial D_c)}<C$ for small enough $\epsilon$, and in the fourth and fifth inequalities above we have used that $||R_\alpha||_{{\cal O}}\leq\frac{1}{2\sqrt{\alpha}}$ (e.g. see \cite{Kirsch-Book96}), and respectively, that $\epsilon < \delta$ and $\psi_0$ satisfies the source condition \eqref{source-cond}. 

Expanding $P$ defined in \eqref{P} and using the fact that $f_{\epsilon,2} = 0$ and $\phi_{\alpha} = T_{\alpha}K^{*}f_{\epsilon} = T_{\alpha}K_{1}^{*}f_{\epsilon,1}$, we obtain
\begin{align}
\label{ec-13}
P & = \frac{2\alpha'}{\|f_{\epsilon,1}\|_{L^{2}(\partial D_{c})}^{2}} \paren{ K_{1}^{*}f_{\epsilon,1},\, T_{\alpha}\phi_{\alpha}}_{L^{2}(\partial D_{a})} - \frac{2\alpha'}{\|f_{\epsilon,1}\|_{L^{2}(\partial D_{c})}^{2}}\paren{K_{1}^{*}K_{1}\phi_{\alpha}, \, T_{\alpha}\phi_{\alpha}}_{L^{2}(\partial D_{a})}\notag\\
& \quad - \frac{2 \alpha'}{2\pi R} \paren{K_{2}^{*}K_{2}\phi_{\alpha},\, T_{\alpha}\phi_{\alpha}}_{L^{2}(\partial D_{a})} + 2w_{i}L_{i}\notag\\
& = \frac{2\alpha'}{\|f_{\epsilon,1}\|^{2}} \paren{T_{\alpha}^{-1}\phi_{\alpha},\,T_{\alpha}\phi_{\alpha}}_{L^{2}(\partial D_{a})} - \frac{2\alpha'}{\|f_{\epsilon,1}\|^{2}}\paren{K_{1}^{*}K_{1}\phi_{\alpha}, T_{\alpha}\phi_{\alpha}}_{L^{2}(\partial D_{a})}\notag\\
& \quad - \frac{2 \alpha'}{2\pi R} \paren{K_{2}^{*}K_{2}\phi_{\alpha},\, T_{\alpha}\phi_{\alpha}}_{L^{2}(\partial D_{a})} + 2w_{i}L_{i}\notag\\
& = \frac{2\alpha'}{\|f_{\epsilon,1}\|^{2}} \paren{(\alpha I + K_{2}^{*}K_{2})\phi_{\alpha},\,T_{\alpha}\phi_{\alpha}}_{L^{2}(\partial D_{a})} - \frac{2 \alpha'}{2\pi R} \paren{K_{2}^{*}K_{2}\phi_{\alpha},\, T_{\alpha}\phi_{\alpha}}_{L^{2}(\partial D_{a})} + 2w_{i}L_{i}\notag\\
& = \frac{2\alpha \alpha'}{\|f_{\epsilon,1}\|^{2}} \paren{\phi_{\alpha},\,T_{\alpha}\phi_{\alpha}}_{L^{2}(\partial D_{a})} + \frac{\alpha'}{\|f_{\epsilon,1}\|^{2}}B\paren{K_{2}^{*}K_{2}\phi_{\alpha}, \, T_{\alpha}\phi_{\alpha}}_{L^{2}(\partial D_{a})} + 2w_{i}L_{i},
\end{align}
where $B = 2 - \displaystyle\frac{\|f_{\epsilon,1}\|_{L^{2}(\partial D_{c})}^{2}}{\pi R}$ and we have used that $T_\alpha^{-1}=\alpha I+K_1^*K_1+K_2^*K_2$ in the third equality above. Observe that \eqref{R-bound} implies $B\geq 0$. Introduce the notation $\widetilde{K}_{2} := \sqrt{B}K_{2}$, and denote $\displaystyle v_{\alpha} := \frac{\phi_{\alpha}}{\|f_{\epsilon,1}\|_{L^{2}(\partial D_{c})}}$. Then \eqref{ec-13} becomes
\begin{align}
\label{ec-14}
P & = 2 \alpha \alpha' \paren{v_{\alpha}, \, T_{\alpha}v_{\alpha}}_{L^{2}(\partial D_{a})} + \alpha' \paren{ \widetilde{K}_{2}^{*}\widetilde{K}_{2}v_{\alpha}, \, T_{\alpha}v_{\alpha}}_{L^{2}(\partial D_{a})} + 2w_{i}L_{i}\notag\\
& = \alpha \alpha' \paren{v_{\alpha}, \, T_{\alpha}v_{\alpha}}_{L^{2}(\partial D_{a})} + \alpha' \paren{ (\alpha I + \widetilde{K}_{2}^{*}\widetilde{K}_{2})v_{\alpha}, \, T_{\alpha}v_{\alpha}}_{L^{2}(\partial D_{a})} + 2w_{i}L_{i}.
\end{align}
From \cite{kato} (see Section V.3.10), for any self-adjoint linear operator $A:H\rightarrow H$, where $H$ is a given Hilbert space (real or complex), we have that:
\begin{equation}
\label{Kato}
0<\gamma = \inf_{\lambda \in \mathrm{Sp}(A)}\lambda \Longrightarrow (Ax, x)_H \geq \gamma(x,x)_H,
\end{equation}
where $(\cdot,\cdot)$ in \eqref{Kato} denotes the usual Hilbert product and where $\mathrm{Sp}(A)$ denotes the real spectrum of the operator $A$. Then, by using \eqref{Kato} for the operator $T_\alpha$ we obtain
\begin{align}
\label{Kato-1}
\paren{ v_{\alpha}, \, T_{\alpha}v_{\alpha} }_{L^{2}(\partial D_{a})} &\geq \frac{1}{\alpha + \mu_{1}^{2}}\|v_{\alpha}\|_{L^{2}(\partial D_{a})}^{2}\geq \frac{1}{1 + \mu_{1}^{2}}\|v_{\alpha}\|_{L^{2}(\partial D_{a})}^{2}\notag\\
&\geq C \|v_{\alpha}\|_{L^{2}(\partial D_{a})}^{2},
\end{align}
where we have used that $\displaystyle \frac{1}{\alpha+\mu_1^2} = \inf_{\lambda \in \mathrm{Sp}(T_\alpha)}\lambda$ with $\mu_1$ denoting the largest singular value of $K$.

Next consider  $\displaystyle D_{\alpha} := \alpha I + \widetilde{K}_{2}^{*}\widetilde{K}_{2}$. Then, because $D_\alpha$ and $T_\alpha$ are linear, bounded, self-adjoint, invertible and positive definite operators, we have that $D_\alpha T_{\alpha}$ will also be linear, bounded, self-adjoint, invertible and have strictly positive eigenvalues. Indeed, for any eigenvalue-eigenvector pair $(x,\lambda)$ of $D_\alpha T\alpha$ we have
\begin{equation*}
D_\alpha T_\alpha x=\lambda x\Rightarrow T_\alpha x=\lambda D_\alpha^{-1}x\Rightarrow \lambda=\frac{( T_\alpha x,x)}{( D^{-1}_\alpha x,x)}\geq 0.
\end{equation*}
Observing that $0\notin \mathrm{Sp}(D_\alpha T_\alpha)$ we have the claim, and the positive definiteness of $D_\alpha T_\alpha$ follows. Thus we have
\begin{equation}
\label{ec-15}
\paren{ (\alpha I + \widetilde{K}_{2}^{*}\widetilde{K}_{2})v_{\alpha}, \, T_{\alpha}v_{\alpha}}_{L^{2}(\partial D_{a})} = \paren{ D_{\alpha}v_{\alpha},\, T_{\alpha}v_{\alpha}}_{L^{2}(\partial D_{a})}\geq 0
\end{equation}
%Then, by using again \eqref{Kato} in \eqref{ec-15} similarly as in \eqref{Kato-1} we obtain
%\begin{align}
%\label{ec-16}
%\paren{ (\alpha I + \widetilde{K}_{2}^{*}\widetilde{K}_{2})v_{\alpha}, \, T_{\alpha}v_{\alpha}}_{L^{2}(\partial D_{a})} & \geq \displaystyle\frac{1}{(\alpha+\mu_1^2)(\alpha+\chi_1^2)}\paren{ u_{\alpha},\, u_{\alpha}}_{L^{2}(\partial D_{a})}\notag\\
%& = \displaystyle\frac{1}{(\alpha+\mu_1^2)(\alpha+\chi_1^2)}\paren{ D_\alpha v_{\alpha},\, D_\alpha v_{\alpha}}_{L^{2}(\partial D_{a})}\notag\\
%& \geq  \displaystyle\frac{\alpha^2}{(1+\mu_1^2)(1+\chi_1^2)}\paren{ v_{\alpha},\, v_{\alpha}}_{L^{2}(\partial D_{a})}\notag\\
%& \geq\displaystyle C\alpha^2\|v_{\alpha}\|_{L^{2}(\partial D_{a})}^{2}
%\end{align}

From  \eqref{ec-14}, \eqref{Kato-1} and \eqref{ec-15} used in \eqref{ec-15'}, we obtain
\begin{align}
\label{ec-17}
|2w_{i}L_{i}| + |Q| & \geq |\alpha'| \left| \alpha \paren{v_{\alpha}, \, T_{\alpha} v_{\alpha}}_{L^{2}(\partial D_{a})} + \paren{ D_{\alpha}v_{\alpha}, \, T_{\alpha}v_{\alpha}}_{L^{2}(\partial D_{a})} \right|\notag\\
& \geq |\alpha'| C\alpha\|v_{\alpha}\|_{L^{2}(\partial D_{a})}^{2} \notag\\
& \geq C \alpha |\alpha'| \|v_{\alpha}\|_{L^{2}(\partial D_{a})}^{2}.
\end{align}
From \eqref{Q} and elementary algebraic manipulations we obtain,
\begin{equation}
\label{ec-18}
|Q| \leq \frac{2 \|f_{1}\|_{L^{2}(\partial D_{c})}}{\|f_{\epsilon, 1}\|_{L^{2}(\partial D_{c})}^{4}} \|f_{\epsilon, 1}\|_{L^{2}(\partial D_{c})}\cdot \delta^{2} \|f_{\epsilon,1}\|_{L^{2}(\partial D_{c})}^{2} = \frac{2 \|f_{1}\|_{L^{2}(\partial D_{c})} \delta^{2}}{\|f_{\epsilon, 1}\|_{L^{2}(\partial D_{c})}} \leq C\delta^{2}
\end{equation}
Recalling that $\displaystyle v_{\alpha} := \frac{\phi_{\alpha}}{\|f_{\epsilon,1}\|_{L^{2}(\partial D_{c})}}$, Lemma \ref{lemma:phif1constant} implies
\begin{equation}
\label{ec-19}
\|v_{\alpha}\|_{L^{2}(\partial D_{a})} \geq C.
\end{equation}
Then from \eqref{P4}, \eqref{ec-18}, and \eqref{ec-19} used in \eqref{ec-17} we finally obtain the statement of the Lemma:
\begin{equation}
\label{ecuatie-4}
\alpha |\alpha'|
 \leq C \delta^{2}+ C\frac{\delta^2}{\sqrt{\alpha}}\leq C\frac{\delta^2}{\sqrt{\alpha}}, \mbox{ for } \epsilon<\epsilon_0.
\end{equation}
\end{proof}

\begin{Arem}
Note that $\|s_\epsilon\|_{L^{2}(\partial D_{c})}\leq (1+\frac{\epsilon}{2})\|f_{1}\|_{L^{2}(\partial D_{c})}$ and thus $f_\epsilon$ as defined above satisfies \eqref{random-f}.
\end{Arem}
The next result is a simple consequence of Proposition \ref{theorem:stabilityestimate} and Lemma \ref{teorema-int}. We have,
\begin{corollary}
\label{corolar-1}
Assume the same notation and the same framework as in Proposition \ref{theorem:stabilityestimate}. Assume also that for $d_c=diam(D_c)$ small enough there exists $d=dist(\partial D_c,\partial D_a)$ small enough so that $\alpha(0)=\alpha_0\approx\delta$. Then we have
\begin{equation}
\displaystyle\frac{\| \phi_{\alpha_\epsilon} - \phi_{\alpha_0}\|_{L^{2}(\partial D_{a})}}{\|\phi_{\alpha_\Ge}\|_{L^{2}(\partial D_{a})}} \leq C\sqrt{\epsilon}.
\end{equation}
\end{corollary}
\begin{proof}
 From Lemma \ref{teorema-int} we obtain that
\begin{equation}
\label{estimare}
|\alpha_\epsilon'| \alpha_\epsilon^{\frac{3}{2}}\leq C\delta^2.
\end{equation}
\noindent Estimate \eqref{estimare} and Cauchy's theorem implies that
\begin{equation}
\label{ecuatie-5}
\left|\frac{\alpha_\epsilon^{\frac{5}{2}}}{\alpha_0^{\frac{5}{2}}}-1\right|= \frac{5}{2}\epsilon\left|\alpha'_{\epsilon_*}\alpha_{\epsilon_*}^{\frac{3}{2}}\alpha_0^{-\frac{5}{2}}\right|\leq C\epsilon\delta^2\alpha_0^{-\frac{5}{2}}\leq   C\sqrt{\epsilon},
\end{equation}
where $\epsilon_*\in(0,\epsilon)$ and we used that $\epsilon<\delta$. Next, simple algebraic manipulation and \eqref{ecuatie-5} imply
\begin{equation}
\left|\frac{\alpha_\epsilon}{\alpha_0}-1\right|\leq \left|\frac{\alpha_\epsilon^5}{\alpha_0^5}-1\right|\leq C\sqrt{\epsilon}\left(\frac{\alpha_\epsilon^{\frac{5}{2}}}{\alpha_0^{\frac{5}{2}}}+1\right)\leq C\sqrt{\epsilon}(2+ C\sqrt{\epsilon})\leq C\sqrt{\epsilon}.
\end{equation}

This together with Proposition \ref{theorem:stabilityestimate} imply the statement of the Corollary.
\end{proof}

\begin{Arem}
We note that all the above results can be adapted to three dimensions. The proof follows exactly the same steps by considering the natural extension of the definition of the discrepancy function $E$ in three dimensions.
\end{Arem}

\begin{Arem}
The assumption made in Corollary \ref{corolar-1} that $\alpha_0\approx\delta$ for $d_c=diam(D_c)$ and $dist(\partial D_c,\partial D_a)=d$ small enough is based on \eqref{ub-alpha} of Lemma \ref{lemma:phif1constant} and on the numerical results presented in Section \ref{sec:numerics}. In particular, for the same settings as in Figure \ref{fig:nonincreasing_F}, Figure \ref{fig:besselFS_alpha0_vary_k_mu} shows that given small $d_c$ for small enough $d$, we have roughly that $10^{-2.5} \leq \alpha_{0} \leq 10^{-1}$. Since $\delta = 2\cdot10^{-2}$ we have in this case that $\frac{1}{2\sqrt{5}} \delta \lesssim \alpha_{0} \lesssim 5 \delta$.
\end{Arem}

\begin{Arem}
\label{constant}
As suggested by our numerical results in Section \ref{sec:numerics} we beleive that the constants (denoted by $C$) in Proposition \ref{theorem:stabilityestimate}, Lemma \ref{teorema-int} and Corollary \ref{corolar-1} will only have small values for $d_c=diam(D_c)$ and $dist(\partial D_c,\partial D_a)=d$ small enough.
\end{Arem}

\begin{figure}
\includegraphics[trim = 20 20 20 20, clip = true, width=\textwidth]{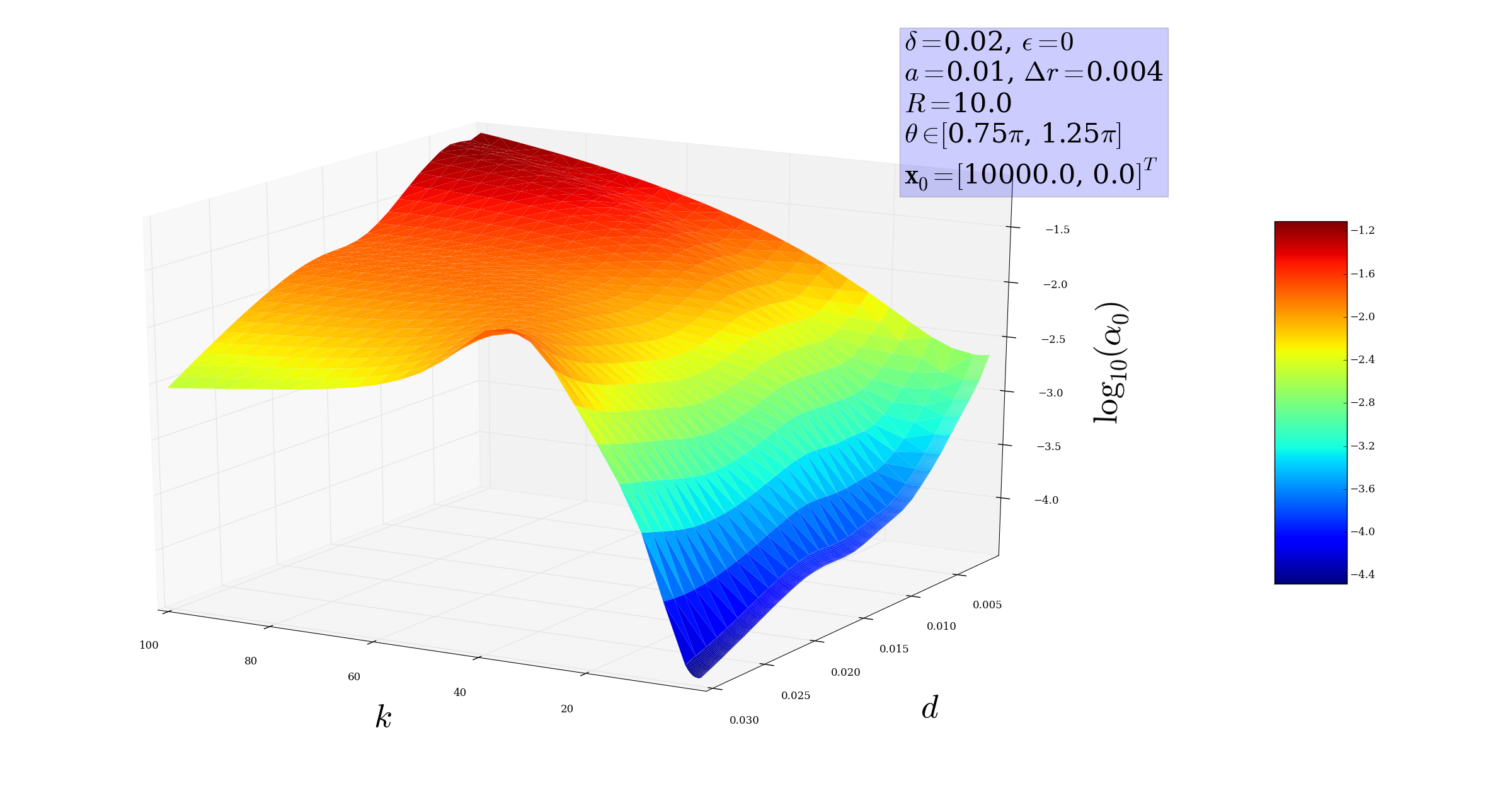}
\caption{Plot of $\alpha_{0}$ with respect to $k$ and $\mu$ with $\delta = 0.02$.}\label{fig:besselFS_alpha0_vary_k_mu}
\end{figure}

\section{Numerics}
\label{sec:numerics}
In this section we proceed with the numerical study of the minimal norm solution for \eqref{Morozov0} obtained through Tikhonov regularization with the Morozov discrepancy principle for the choice of the regularization parameter in two dimensions. First we focus on the general setup of our numerical approach, and then in Section \ref{subsec:parametersetup} we discuss more specifically the parameters used in our numerical examples. In Sections \ref{subsec:nearfieldstability} and \ref{subsec:sphericalsource} we present numerical data which demonstrates how stably $\phi$ depends on $f$ and various control statistics for a spherical point source. All figures generally display their respective parameters in an offset legend.

For all of the numerical computations done in this section, we discretize the integral operator $K$ via the method of moment collocation. We refer to (\cite{Kress-Book99}, \S 17.4) for more details on the method. First we choose an approximate basis of functions for $L^{2}(\partial D_{a})$. To do this we suppose the domain $D_{a}$ can be parametrized in polar coordinates by points
\begin{equation*}
(s(\tau)\cos{\tau}, \, s(\tau)\sin{\tau})), \quad \tau \in [0,2\pi],
\end{equation*}
where $s:\mathbb{R} \to \mathbb{R}_{+}$ is a $2\pi$-periodic smooth function. Using these coordinates, any function $\phi$ defined on $\partial D_{a}$ can be realized via the pullback as a function of $\tau$:
\begin{equation*}
\phi(s(\tau)\cos{\tau}, \, s(\tau)\sin{\tau}).
\end{equation*}
For convenience, let us use the notation $\widehat{\tau} = (\cos{\tau},\sin{\tau})$ and $\widehat{\tau}^{\perp} = (-\sin{\tau}, \cos{\tau})$.

Now let $n_{a} \in \mathbb{N}$ and let $\tau_{j} = \frac{2\pi j}{n_{a}}$, $0 \leq j \leq n_{a}-1$ be $n_{a}$ equally spaced points on the interval $[0,2\pi)$. We then use the exponential basis functions $\{e^{il\tau}\}_{l=0}^{n_{a}-1}$ for $L^{2}([0,2\pi])$ and approximate a given $\phi \in L^{2}(\partial D_{a})$ via interpolation at the points $\{\tau_{j}\}_{j=0}^{n_{a}-1} \subset [0,2\pi]$. Note that
\begin{align}
\int_{\partial D_{a}}\phi(\mathbf{y}) \frac{\partial \Phi}{\partial \nu_{\mathbf{y}}}(\mathbf{x}, \, \mathbf{y})\,dS_{\mathbf{y}} & = \int_{0}^{2\pi}\phi(s(\tau)\cos{\tau}, \, s(\tau)\sin{\tau})\frac{\partial \Phi}{\partial \nu_{y}}(\mathbf{x}, (s(\tau)\cos{\tau}, \, s(\tau)\sin{\tau})) \notag\\
& \quad \cdot \sqrt{s(\tau)^{2} + s'(\tau)^{2}}\,d\theta.
\end{align}
Furthermore, since $\left( s'(\tau)\cos{\tau} - s(\tau)\sin{\tau}, \, s(\tau)\cos{\tau} + s'(\tau)\sin{\tau}\right)$ is a tangent vector to $\partial D_{a}$, we have that
\begin{align*}
\nu(\mathbf{y}) = \nu(\tau) & = \frac{(s(\tau)\cos{\tau} + s'(\tau)\sin{\tau}, \, s(\tau)\sin{\tau} - s'(\tau)\cos{\tau})}{\sqrt{ s(\tau)^{2} + s'(\tau)^{2}}}\\
& = \frac{s(\tau)\widehat{\tau} - s'(\tau)\widehat{\tau}^{\perp}}{\sqrt{ s(\tau)^{2} + s'(\tau)^{2}}}
\end{align*}
is the unit outward normal vector to $\partial D_{a}$. It is then straightforward to compute in the case of the Helmholtz equation in 2D that
\begin{align*}
& \quad \frac{\partial \Phi}{\partial \nu_{\mathbf{y}}}(\mathbf{x}, \, (s(\tau)\cos{\tau}, \, s(\tau)\sin{\tau}))\\
& = \nabla_{y}\Phi(\mathbf{x}, (s(\tau)\cos{\tau}, \, s(\tau)\sin{\tau})) \cdot \nu(\tau)\\
& = \frac{ik}{4}H_{0}^{(1)'}(k|\mathbf{x} - s(\tau)\widehat{\tau}|)\frac{s(\tau)\widehat{\tau} - \mathbf{x}}{\sqrt{s(\tau)^{2} + |\mathbf{x}|^{2} - 2s(\tau)\mathbf{x}\cdot \widehat{\tau}}} \cdot \frac{s(\tau)\widehat{\tau} - s'(\tau)\widehat{\tau}^{\perp}}{\sqrt{s(\tau)^{2} + s'(\tau)^{2}}}.
\end{align*}

Let $n_{a} \in \mathbb{N}$ be the number of sample points on the antenna, $\partial D_{a}$, and let $n_{c} \in \mathbb{N}$ be the total number of sample points on $\partial D_{c}$. Also let $n_{R}$ be the total number of sample points on $\partial B_{R}$. We write the $2 \times (n_{c} + n_{R})$ matrix of points
\begin{equation*}
\mathbf{X} := [x_{1}, \ldots, x_{n_{c}+n_{R}}],
\end{equation*}
where each $x_{j}$ is a $2$-vector, $\{x_{j}\}_{j=1}^{n_{c}} \subset \partial D_{c}$ and $\{x_{j}\}_{j=n_{c}+1}^{n_{c} + n_{R}} \subset \partial B_{R}$. Approximations of all the integrals involved are then computed using a standard left endpoint sum with the appropriate quadrature weights. All the numerical examples presented herein take $D_{c}$ to be an annular sector parametrized by $r \in [r_{1}, r_{2}]$ and $\theta \in [\theta_{1}, \theta_{2}]$. See Figure \ref{fig:numericssetup} for details.

\begin{figure}
\centering
\def \svgwidth{0.6\linewidth}
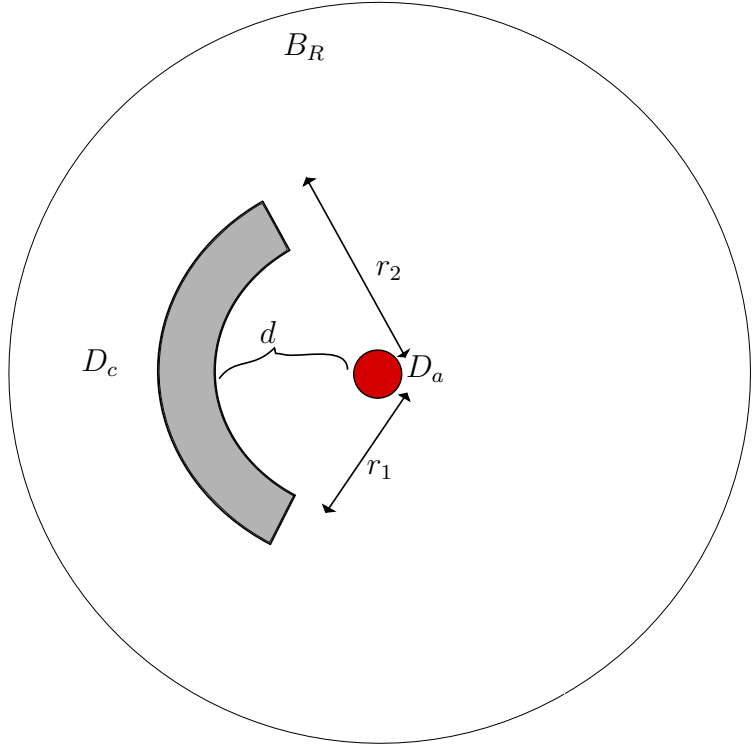
\caption{Antenna and control region geometry used for numerical experiments.}\label{fig:numericssetup}
\end{figure}

For each $1 \leq j \leq n_{c} + n_{R}$ and each $0 \leq l \leq n_{a}-1$, we compute $K[e^{il\tau}](x_{j})$ via the approximation
\begin{equation*}
\frac{2\pi}{n_{a}}\sum_{m=0}^{n_{a}-1} \frac{\partial \Phi( x_{j}, [s(\tau_{m})\cos\paren{\tau_{m}}, \, s(\tau_{m})\sin{\tau_{m} }]^{T})}{\partial \nu_{\mathbf{y}}} e^{il \tau_{m}}\sqrt{s(\tau_{m})^{2} + s'(\tau_{m})^{2}}.
\end{equation*}
If we fix $j$ and vary $l$, we see that the above sum is equivalent to computing the discrete Fourier transform of the $n_{a}$-vector
\begin{equation}
\mathbf{v}_{j} := \left[\frac{\partial \Phi( x_{j}, [s(\tau_{m})\cos{\tau_{m}}, \, s(\tau_{m})\sin{\tau_{m} }]^{T})}{\partial \nu_{\mathbf{y}}}\sqrt{s(\tau_{m})^{2} + s'(\tau_{m})^{2}}\right]_{m=0}^{n_{a}-1}, \label{eq:vj}
\end{equation}
which can be computed efficiently using the Fast Fourier Transform algorithm. In particular, for the vector $\mathbf{v}_{j}$ in \eqref{eq:vj}
\begin{equation*}
[K\{e^{il\tau}\}(x_{j})]_{l=0}^{n_{a}-1} \approx 2\pi \mathsf{FFT}(\mathbf{v}_{j}),
\end{equation*}
where $\mathsf{FFT}$ is defined on $n_{a}$-vectors $\mathbf{w} = [w_{1},\ldots, w_{n_{a}}]^{T} \in \mathbb{C}^{n_{a}}$ by
\begin{equation}
\mathsf{FFT}(\mathbf{w}) = \left[ \frac{1}{n_{a}}\sum_{j=1}^{n_{a}}w_{j}e^{\frac{2\pi i (j-1)(l-1)}{n_{a}}} \right]_{l=1}^{n_{a}}.
\label{eq:fft}
\end{equation}
So the matrix representation of $K$ is then the $n_{a} \times (n_{c} + n_{R})$ matrix
\begin{equation}
A := 2\pi [\mathsf{FFT}(\mathbf{v}_{1}), \cdots, \mathsf{FFT}(\mathbf{v}_{n_{c}+n_{R}})].
\end{equation}

Now, in order to approximately solve the ill-posed problem $K\phi = (f_{1}, f_{2})$, we attempt to solve the linear system
\begin{align*}
K_{1}\phi(x_{j}) & = f_{1}(x_{j}), \quad 1 \leq j \leq n_{c}\\
K_{2}\phi(x_{j}) & = f_{2}(x_{j}), \quad n_c + 1 \leq j \leq n_c + n_{R}.
\end{align*}
Since $A$ is computed with respect to the functions $e^{il\theta}$, we first consider the approximate coefficients of $\phi$ with respect to the finite basis $\{e^{il\tau}\}_{l=0}^{n_{a}-1}$, given by
\begin{equation}
c_{l} := \frac{1}{n_{a}}\sum_{m=0}^{n_{a}-1}e^{-il\tau_{m}}\phi(s(\tau_{m})\cos(\tau_{m}), \, s(\tau_{m})\sin(\tau_{m}))\,d\tau \approx \frac{1}{2\pi}\int_{0}^{2\pi} e^{-il\tau}\phi(s(\tau)\cos(\tau), \, s(\tau)\sin(\tau))\,d\tau.
\end{equation}
Let
\begin{equation*}
h = [c_{0}, c_{1}, \ldots, c_{n_{a} - 1}]^{T} \in \mathbb{C}^{n_{a}}.
\end{equation*}

We then numerically compute the Tikhonov regularized solution
\begin{equation*}
h_{\alpha} := (A^{*}A + \alpha I)^{-1}A^{*}f,
\end{equation*}
with $\alpha > 0$. The solution vector $h_{\alpha}$ yields the approximate coefficients $c_{l}$ of the desired density $\phi$ with respect to the functions $\{e^{il\tau}\}_{l=0}^{n_{a}-1}$. We obtain the density $\phi_{\alpha}$ corresponding to $h_{\alpha}$ sampled at the angles $\tau_{m}$ on $\partial D_{a}$ by the formula
\begin{equation*}
\phi_{\alpha}(\tau_{m}) := \sum_{l=0}^{n_{a}-1}[h_{\alpha}]_{l}e^{il\tau_{m}}.
\end{equation*}

After computing the residual $K\phi - f$ (e.g. for $\phi = \phi_{\alpha}$), we will then need to compute
\begin{equation*}
E(\phi_\alpha, f) = \frac{1}{\|f_{1}\|_{L^{2}(\partial D_{c})}^{2}}\| K_{1}\phi  - f_{1}\|_{L^{2}(\partial D_{c})}^{2} + \frac{1}{2\pi R} \|K_{2}\phi - f_{2}\|_{L^{2}(\partial B_{R})}^{2}.
\end{equation*}
Recall that the discrepancy function $F(\alpha)$ was defined by 
\begin{equation}
F(\alpha) = E(\phi_\alpha, f) - \delta^{2}, \label{eq:F}
\end{equation}
where $\delta > 0$ is a fixed error parameter. As discussed in Section \ref{sec:stability}, the mapping 
\begin{equation*}
\alpha \mapsto E(\phi_\alpha, f) 
\end{equation*}
is not globally increasing, as can be numerically demonstrated. However, for certain feasible regions of $(\alpha, \epsilon)$, $F$ is increasing. And in this case, there is a unique $\alpha_{\delta}$ such that $F(\alpha_{\delta}) = 0$. To find $\alpha_{\delta}$, we use Newton's method combined with an initial coarse line search to identify a good starting point.

First note that if we split the matrix $A$ into two blocks $A_{near}$ ($n_{c}$ by $n_{a}$) and $A_{far}$ ($n_{R}$ by $n_{a}$) so that
\begin{equation*}
A = \left[ \begin{array}{c}
A_{near}\\
A_{far}
\end{array}\right],
\end{equation*}
then $[A\phi]_{1} = A_{near}\phi$, $[A\phi]_{2} = A_{far}\phi$, and $A^{*}A = A_{near}^{*}A_{near} + A_{far}^{*}A_{far}$. In the discretized setting, instead of (\ref{eq:F1}) we take
\begin{equation}
F(\alpha) =\frac{1}{\|f_{1}\|^{2}} \|A_{near}h_{\alpha} - f_{1}\|_{L^{2}(\partial D_{c})}^{2} + \frac{1}{2\pi R}\|A_{far}h_{\alpha} - f_{2}\|_{L^{2}(\partial B_{R})}^{2} - \delta^{2}
\end{equation}
with
\begin{equation}
h_{\alpha} = (A^{*}A + \alpha I)^{-1}A^{*}f = (A_{near}^{*}A_{near} + A_{far}^{*}A_{far} + \alpha I)^{-1}\paren{A_{near}^{*}f_{1} + A_{far}^{*}f_{2}}.
\end{equation}
Then in the same spirit as that presented in \cite{CoKr-Book98} for Tikhonov regularization with respect to the standard $L^2$ norm, we compute
\begin{align}
F'(\alpha) & = \frac{-2\alpha}{\|f_{1}\|_{L^{2}(\partial D_{c})}^{2}}\re\paren{\frac{\partial h_{\alpha}}{\partial \alpha}, \, h_{\alpha}} \notag\\
& \quad + \paren{ \frac{1}{\pi R} - \frac{2}{\|f_{1}\|_{L^{2}(\partial D_{c})}^{2}}}\re\paren{\frac{\partial h_{\alpha}}{\partial \alpha}, \, A_{far}^{*}A_{far}h_{\alpha} - A_{far}^{*}f_{2}} \label{eq:dF1}\\
\frac{\partial h_{\alpha}}{\partial \alpha} & = -(A^{*}A + \alpha I)^{-1}h_{\alpha}, \label{eq:dF2}
\end{align}
where $( \cdot, \cdot )$ denotes the $L^{2}$ inner product on $\partial D_{a}$.

The function $f_{1}$ defined on $\partial D_{c}$ could be, for example, the trace of a plane wave, or of the fundamental solution to the Helmholtz equation based at some fixed point $\mathbf{x}_{0}$, i.e. a point source. For this paper, we focus on the case where $f_{1}$ is a point source and where $f_{2} \equiv 0$ on $\partial B_{R}$. A spherical point source is represented as
\begin{equation}
\frac{i}{4}H_{0}^{(1)}(k|\mathbf{x} - \mathbf{x}_{0}|), \label{eq:sphericalsource}
\end{equation}
where $\mathbf{x}_{0}$ is the source point (typically outside of $B_{R}$).

For such an $f_{1}$, there are some quantities in which we will be interested so as to determine the effectiveness of a given density $\phi$ in solving the problem $K\phi =f$. These are: the relative error of $K\phi$ on $\partial D_{c}$; the $L^{2}$ average of $K\phi$ on $\partial B_{R}$; the relative and absolute stability of $\phi$ when applying a small perturbation to $f_{1}$; the norm of $\phi$ on $\partial D_{a}$. In other words, we will measure
\begin{equation}
\frac{\|K_{1} \phi - f_{1}\|_{L^{2}(\partial D_{c})}}{\|f_{1}\|_{L^{2}(\partial D_{c})}}, \quad \frac{1}{\sqrt{2\pi R}} \|K_{2}\phi\|_{L^{2}(\partial B_{R})}, \label{eq:computevar1}
\end{equation}
\begin{equation}
\frac{\|\phi_{\alpha_{\epsilon}}- \phi_{\alpha_0}\|_{L^{2}(\partial D_{a})}}{\|\phi_{\alpha_0}\|_{L^{2}(\partial D_{a})}}, \quad \|\phi_{\alpha_\epsilon} - \phi_{\alpha_0}\|_{L^{2}(\partial D_{a})}, \label{eq:computevar2}
\end{equation}
and
\begin{equation}
\|\phi\|_{L^{2}(\partial D_{a})}, \label{eq:phinorm}
\end{equation}
where $\phi_{\alpha_\epsilon}$ is the Tikhonov regularization solution to $K\phi = (f_{1,\epsilon},0)$ with $\|f_{1} - f_{1,\epsilon}\|_{L^{2}(\partial D_{c})} = \epsilon \|f_{1}\|_{L^{2}(\partial D_{c})}$, and $\phi_{\alpha_0}$ is the solution with unperturbed $f_{1}$. The Morozov solution $\alpha_0$ and $\alpha_\epsilon$ are computed via Newton's Method using \eqref{eq:dF1} and \eqref{eq:dF2} such that 
\begin{align}
E(\phi_{\alpha_0}, f)  & = \delta^2 \notag\\
E(\phi_{\alpha_\epsilon}, f_{\epsilon}) & = \delta^2.
\end{align}
See also the discussion following \eqref{eq:F}. Recall from (\ref{random-f}), that when adding noise to the data $(f_{1}, 0)$, we choose a random perturbation $\eta \in L^{2}(\partial D_{c})$ and set
\begin{equation}
f_{1,\epsilon} = f_{1} + \epsilon \widehat{\eta} \|f_{1}\|_{L^{2}(\partial D_{c})}, \label{eq:noiseterm}
\end{equation}
where $\epsilon > 0$ represents the relative percentage of noise added. In the discrete case, the noise is chosen to be a complex $n_c$-vector $\nu$ whose real and imaginary components are pseudorandom numbers (we used uniformly distributed noise, but any distribution would yield similar results) on the interval $(-1,1)$. Furthermore, for reproducibility, whenever generating $\nu$ using a pseudorandom number generator, we always reset the seed to the same value.

\subsection{Parameters Used for Numerical Experiments\label{subsec:parametersetup}}
Here we describe some of the parameters used for the various numerical experiments presented. In Section \ref{subsec:sphericalsource} we always assume that $\partial D_{a}$ is a circle with radius given by $a = 0.01$, and that $\partial D_{c}$ is a sector of an annulus with $\theta_{1} = 3\pi/4$ and $\theta_{2} = 5\pi/4$. We also take $R = 10$ in all computational examples. We remark that we always restrict the distance from $D_{c}$ to $D_{a}$ to be no smaller than $10^{-3}$ due to the numerical limitations of our approach. This is due to the fact that $K\phi$ is a singular integral when evaluating at points very near to $\partial D_{a}$. Therefore, at points on $\partial D_{c}$ that are near $\partial D_{a}$ the limitations of machine precision become more and more apparent. Numerically, we observed that our direct approach starts to break down near $d = \mathrm{dist}(\partial D_{c}, \partial D_{a}) = 10^{-4}$. However, we stress that one could most likely obtain high accuracy in computing $K\phi$ for $d \leq 10^{-4}$ by using the Nystr\"om method as discussed in \cite{Kress-Book99}.

For the collocation method, we use $n_{a} = 256$ sample points on $\partial D_{a}$, and $n_{\mathrm{arc}_{1}} = 256$ points on the inner arc of $\partial D_{c}$, with the remaining points chosen so as to keep the quadrature weights approximately constant. Thus for a very thin region, $n_{c} \approx 512$. We also take $n_{R} = 256$ (number of sample points on $\partial B_{R}$). Note that increasing $n_{c}$ or $n_{R}$ will put more emphasis on matching $f$ on $\partial D_{c}$ or $\partial B_{R}$, respectively. The discrepancy parameter $\delta$ used for Tikhonov regularization will typically be fixed at $0.02$. The key variables under consideration are $d = r_{1} - a$, $k$, and $\epsilon$ (perturbation parameter for adding noise to $f_{1}$). All of the plots presented in the following sections involve varying two of the aforementioned parameters and plotting different quantities of interest, as stated in (\ref{eq:computevar1})-(\ref{eq:phinorm}).

When evaluating the relative change in $\phi$ given a perturbation of $f_{1}$, denoted by $f_{\epsilon, 1}$, we remark that for the parameter choices we used, a $0.5\%$ change ($\epsilon = 0.005$) in $f_{1}$ yielded a roughly $5\%$ change in $\phi$. However, one must keep in mind that this depends quite a lot on the parameters used. In particular, setting the discrepancy $\delta = 0.02$ in all the simulations had an important effect on the numerical results. If we had used $\delta = 0.05$ instead (which leads to approximately a $5 \%$ mismatch on the region $\partial D_{c}$), then the relative change in $\phi$ given $\epsilon = 0.005$ would be noticeably smaller. So ultimately there is a tradeoff between $\mu$, $R$, $k$, $\delta$, and $\epsilon$ which is not entirely trivial, but this still can be examined experimentally as we have done.

\subsection{Near field stability\label{subsec:nearfieldstability}}
We present below Figure \ref{fig:singularvalues}, which shows how the first $50$ singular values of the operator $K = (K_{1}, K_{2})$ vary with $d$. It is clear that for $d$ small (i.e. for control in the nearfield of the antenna), the rate of decay of the singular values of $K$ is considerably slower than for larger $d$. This in turn provides some experimental evidence for the fact that nearfield control seems to be more feasible in terms of stable dependence of the solution $\phi$ on $f$. We also show Figure \ref{fig:singularvaluesdiff}, which shows the behavior of the first and sixth singular value of $K$ with respect to $d$ and $k$.

\begin{figure}
\includegraphics[width=\textwidth]{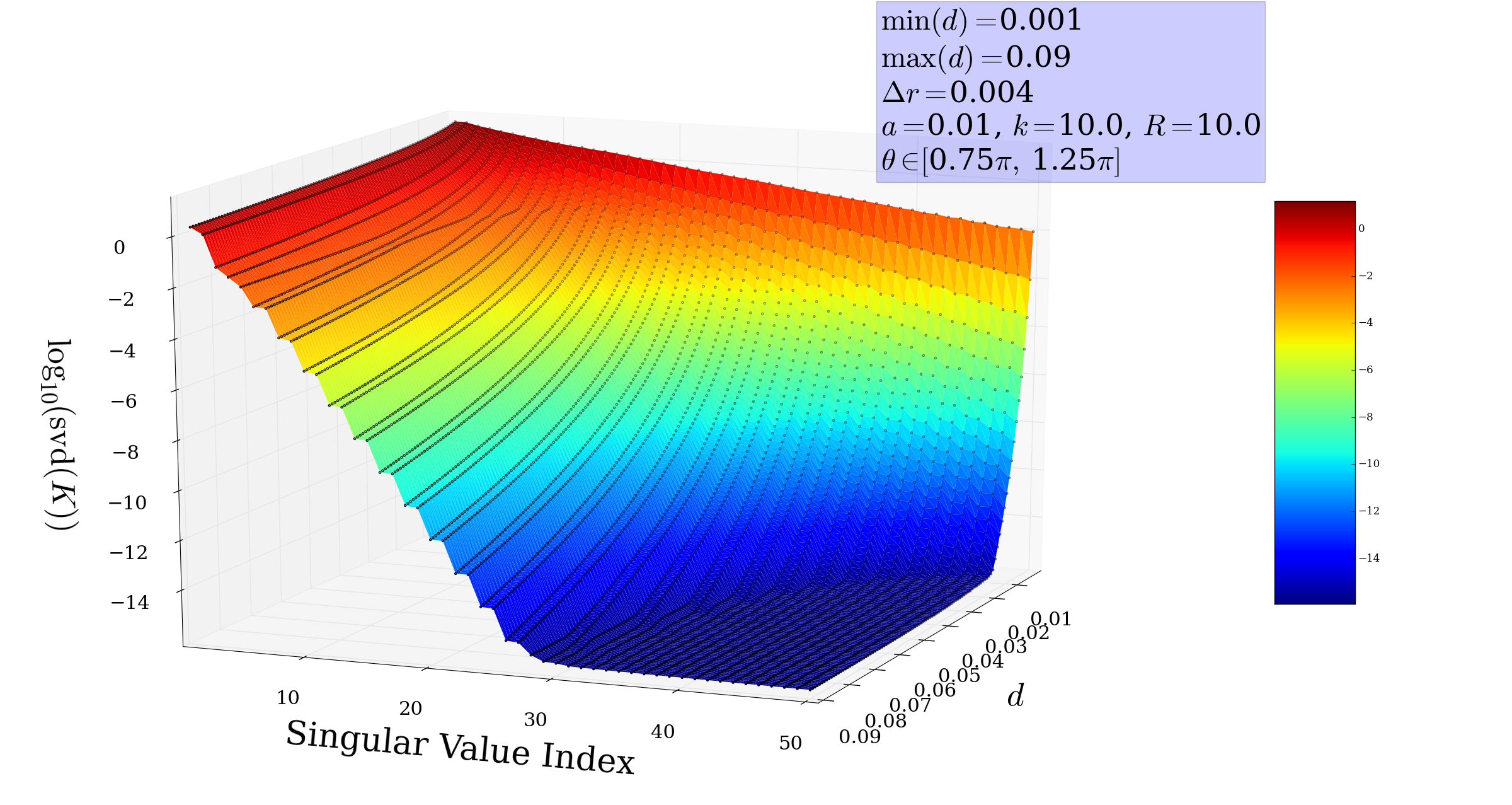}
\caption{Plot of first $50$ singular values of $K:L^{2}(\partial D_{a}) \to \Xi$ for $\partial D_{a}$ a circular antenna of radius $a = 0.01$ and $\partial D_{c}$ an annular region of varying distance from $\partial D_{a}$.}\label{fig:singularvalues}
\end{figure}

\begin{figure}
\includegraphics[width=\textwidth]{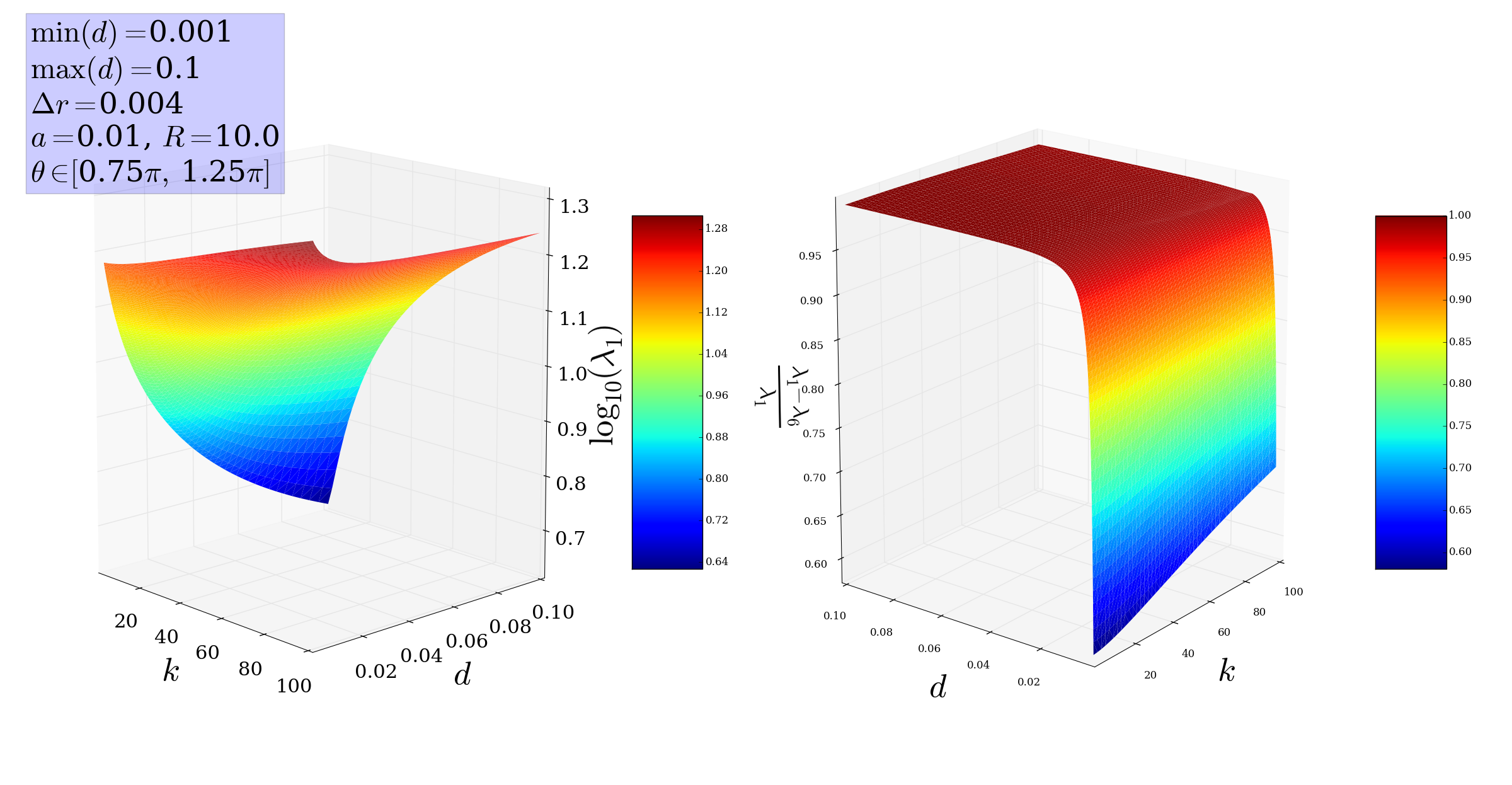}
\caption{Plot of first singular value of $K:L^{2}(\partial D_{a}) \to \Xi$ as well as the relative difference of the first and sixth singular values with respect to $d$ and $k$. Again, $\partial D_{a}$ is a circular antenna of radius $a = 0.01$ and $\partial D_{c}$ an annular region.}\label{fig:singularvaluesdiff}
\end{figure}

\subsection{Control for a Spherical Point Source\label{subsec:sphericalsource}}
We now consider the case that
\begin{equation*}
f_{1}(\mathbf{x}) = \frac{i}{4}H_{0}^{(1)}(k|\mathbf{x} - \mathbf{x}_{0}|),
\end{equation*}
where $\mathbf{x}_{0}$ is the source point. In all examples presented in this section, we have $R = 10$ unless otherwise noted, and $\mathbf{x}_{0} = [20, 0]^{T}$ or $\mathbf{x}_{0} = [10000, 0]^{T}$ (to approximate a source at infinity).

First we observe how the frequency $k$ and distance $d$ from $\partial D_{c}$ to $\partial D_{a}$ affects the various control criteria. In figures \ref{fig:besselFS_epsilon_0005_delta_002_vary_k_01_to_100_mu_0001_to_003_data} and \ref{fig:besselNS_epsilon_0005_delta_002_vary_k_01_to_100_mu_0001_to_003_data} we vary $k$ from $0.1$ to $100$ and $d$ from $0.001$ to $0.003$ with $a = 0.01$. With the error discrepancy set at $\delta = 0.02$, we have in both figures that relative error on $\partial D_{c}$ is very close to $2\%$ for all data points. Moreover, with $0.5\%$ noise added to $f$, roughly a $5\%$ change in $\phi$ is observed at all frequencies when $d$ is near its lower limit. A bit more sensitivity is observed for frequencies $k < 20$ when $d$ increases beyond $0.01$. Interestingly, for $k > 20$ the optimal parameter $\alpha$ is larger and corresponding power budget smaller in order to achieve discrepancy $\delta$.

\begin{figure}
\includegraphics[trim = 15 0 40 0, clip = true, width=\textwidth]{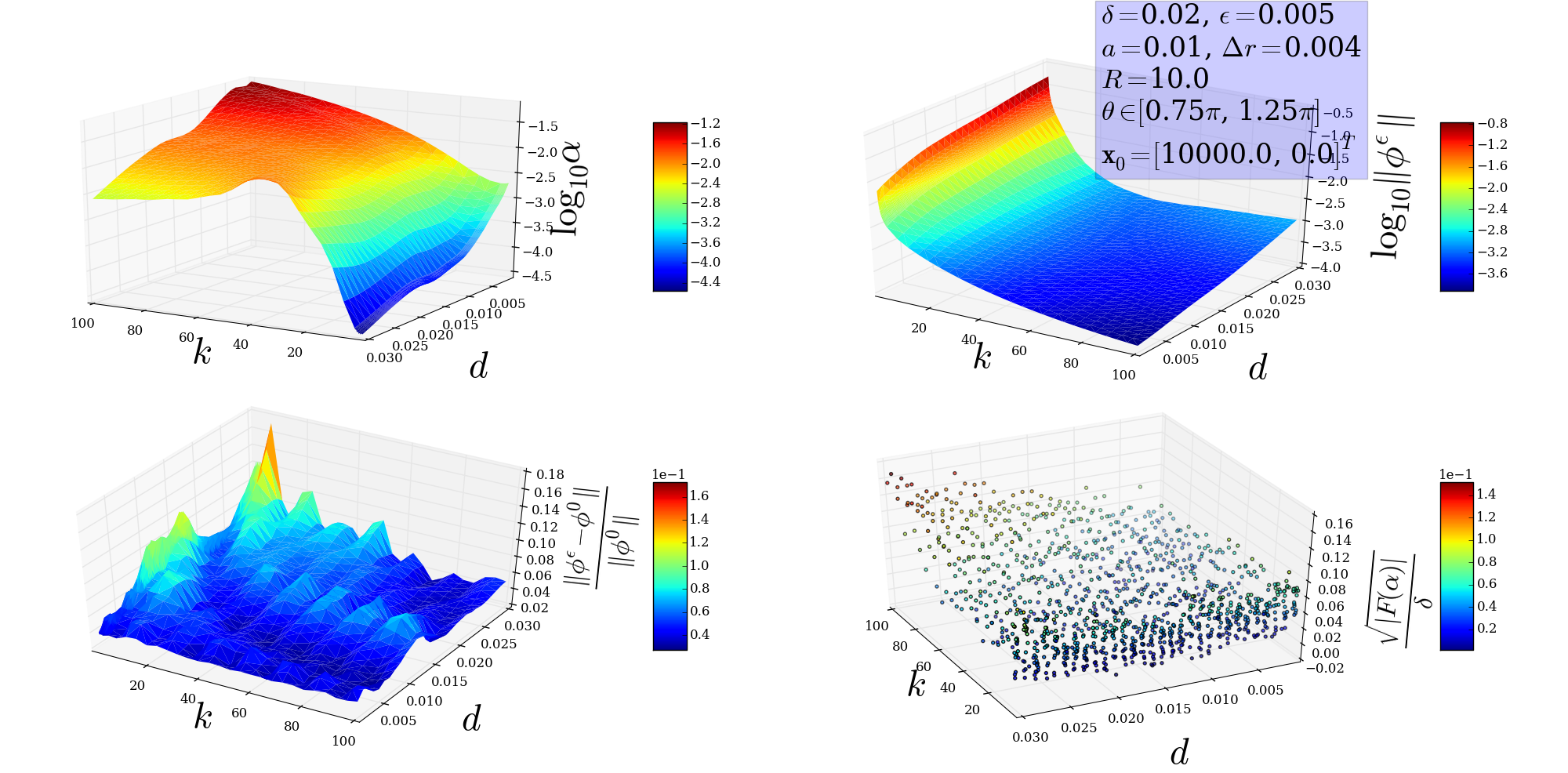}
\caption{Plot vs. $k$ and  $d$ for $f_{1}$ a spherical point source at $\mathbf{x}_{0} = [10000,\, 0]^{T}$.}\label{fig:besselFS_epsilon_0005_delta_002_vary_k_01_to_100_mu_0001_to_003_data}
\end{figure}
\begin{figure}
\includegraphics[trim = 15 0 40 0, clip = true, width=\textwidth]{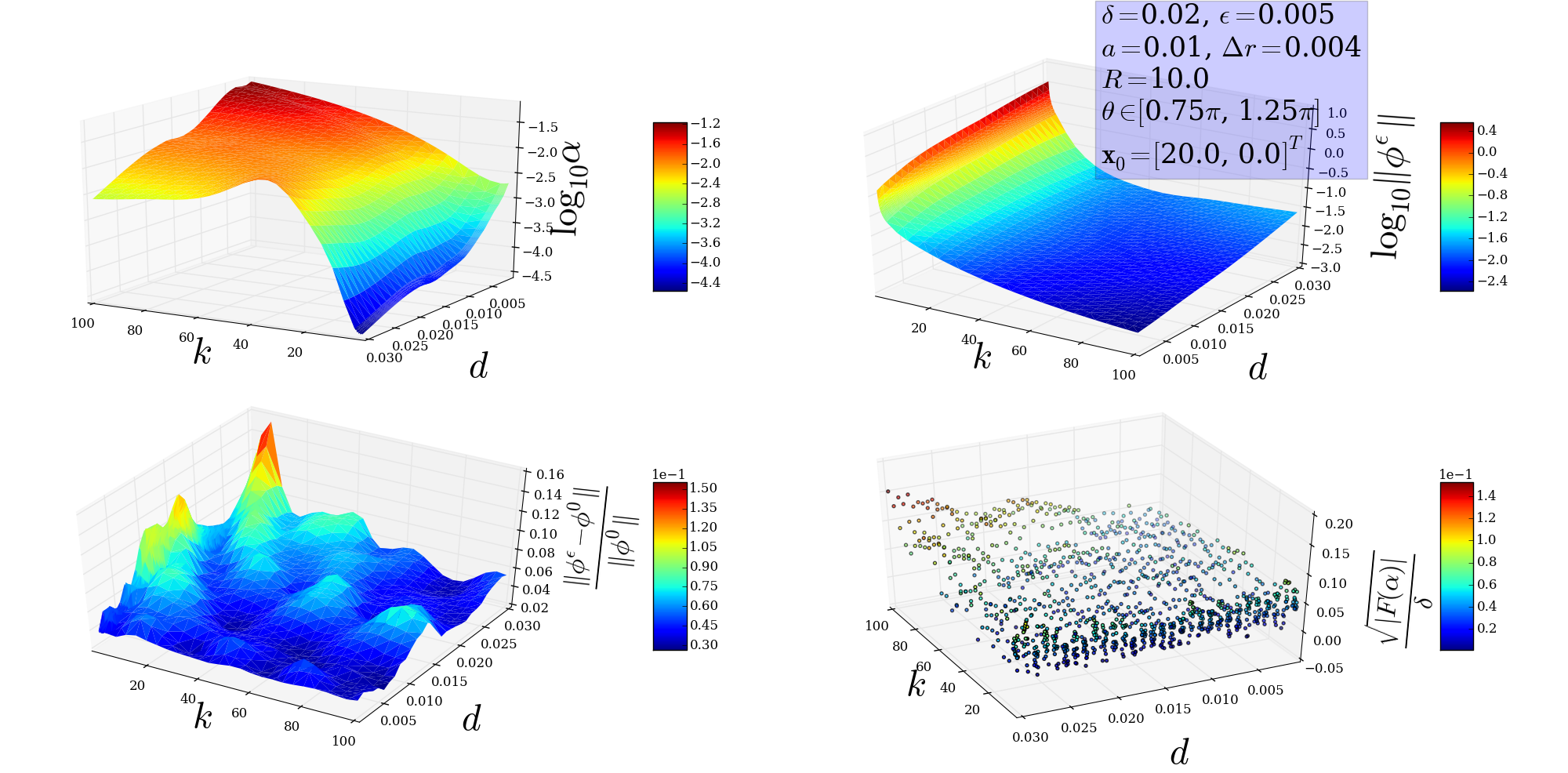}
\caption{Plot vs. $k$ and  $d$ for $f_{1}$ a spherical point source at $\mathbf{x}_{0} = [20,\, 0]^{T}$.}\label{fig:besselNS_epsilon_0005_delta_002_vary_k_01_to_100_mu_0001_to_003_data}
\end{figure}
In figures \ref{fig:besselFS_epsilon_0005_delta_002_vary_k_01_to_100_mu_0001_to_003_data} and \ref{fig:besselNS_epsilon_0005_delta_002_vary_k_01_to_100_mu_0001_to_003_data} it is clear that for smaller $d$ values the sensitivity of $\phi$ to $0.5\%$ noise added to $f$ is close to $5\%$. As $d$ increases, sensitivity of $\phi$ to noise increases as expected. Having $\mathbf{x}_{0}$ nearer or farther from $\partial B_{R}$ does not have a very significant effect on the overall shape of each subplot.

\begin{figure}
\includegraphics[trim = 15 0 40 0, clip = true, width=\textwidth]{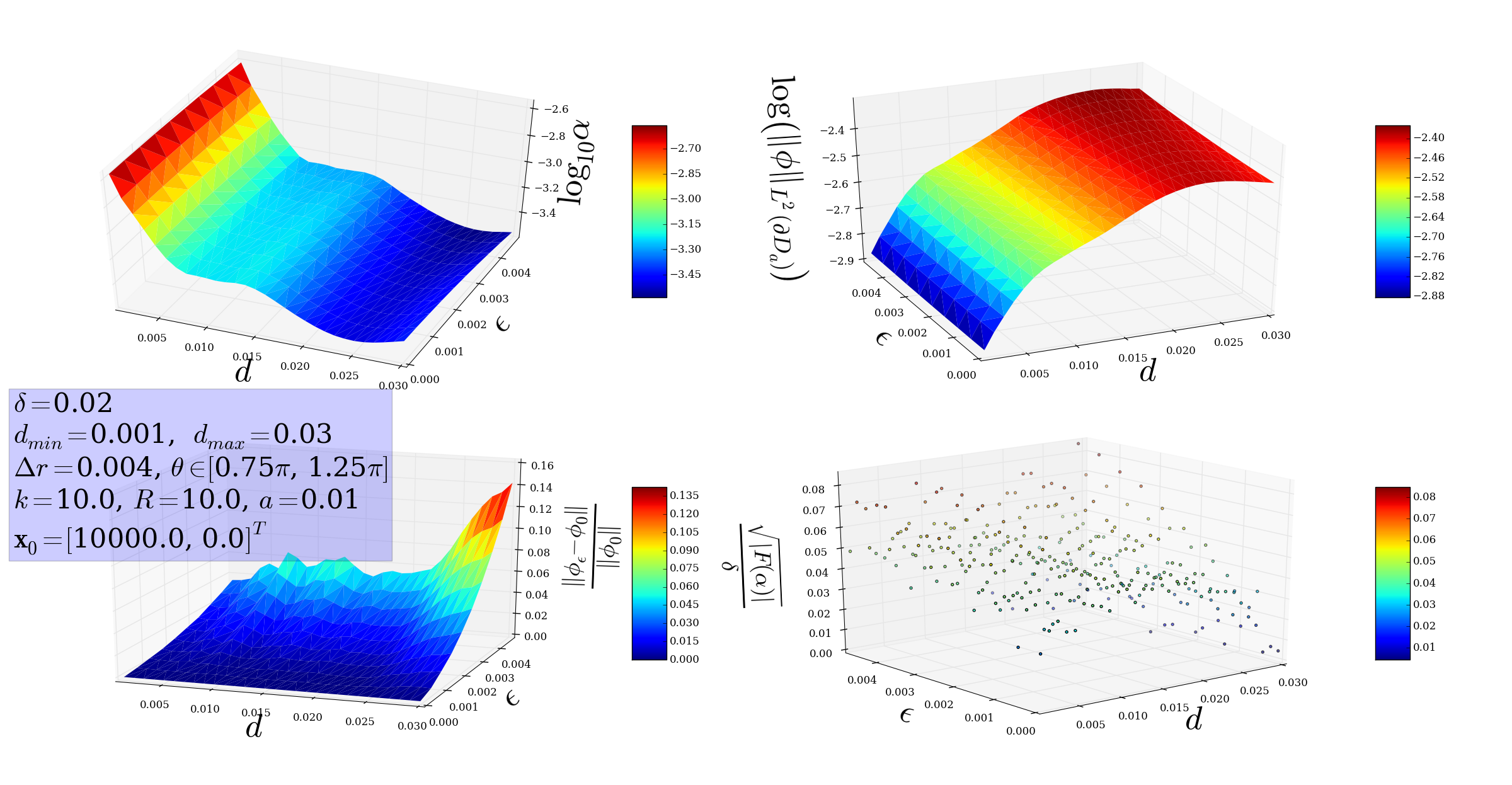}
\caption{Plot vs. $d$ and  $\epsilon$ for $f_{1}$ a spherical point source at $\mathbf{x}_{0} = [10000,\, 0]^{T}$.}\label{fig:besselFS_vary_mu_epsilon}
\end{figure}
\begin{figure}
\includegraphics[trim = 15 0 40 0, clip = true, width=\textwidth]{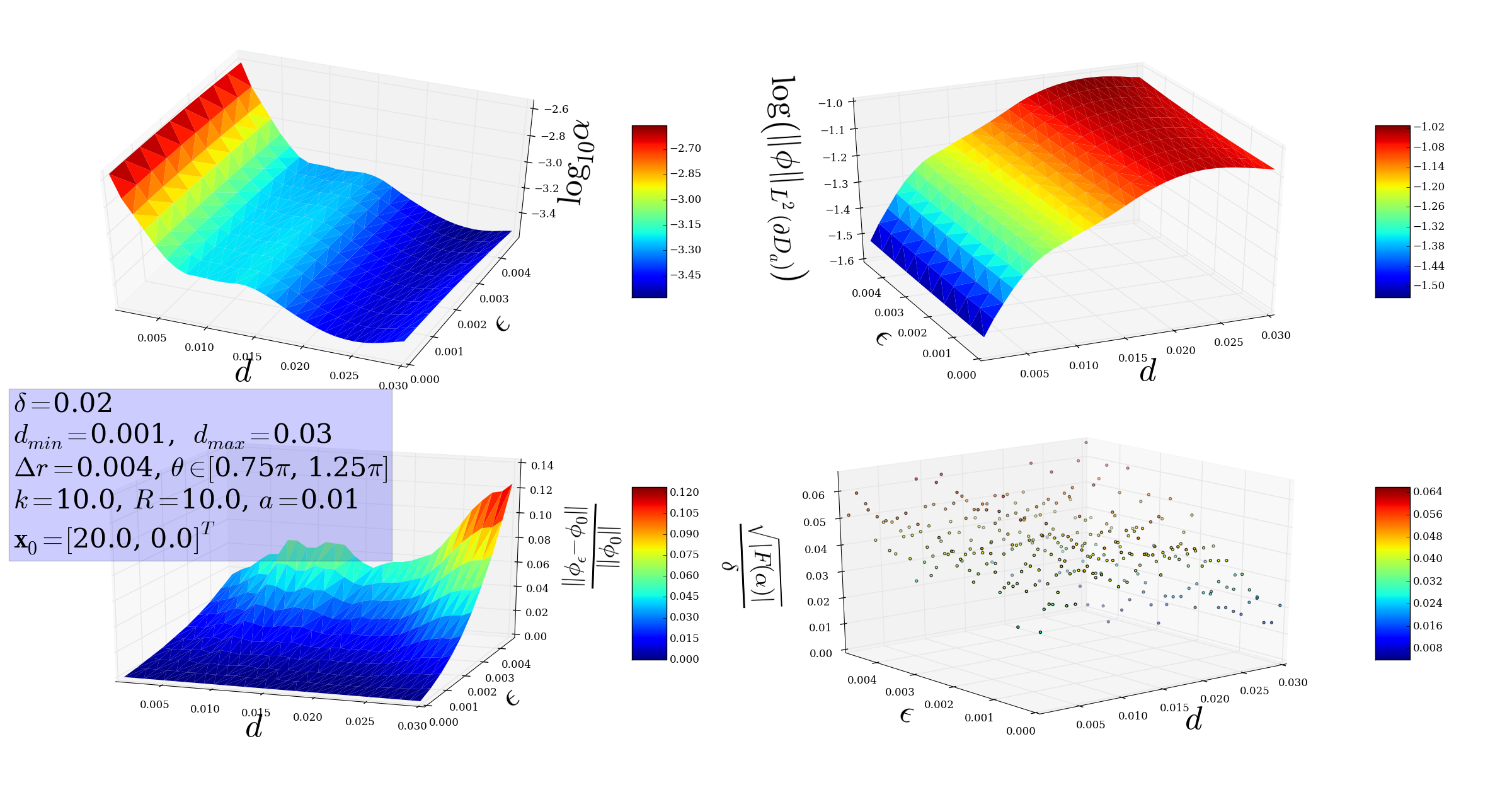}
\caption{Plot vs. $d$ and  $\epsilon$ for $f_{1}$ a spherical point source at $\mathbf{x}_{0} = [20,\, 0]^{T}$.}\label{fig:besselNS_vary_mu_epsilon}
\end{figure}
Figures \ref{fig:besselFS_vary_mu_epsilon} and \ref{fig:besselNS_vary_mu_epsilon} show how the quantities of interest change with $d$ and the noise factor $\epsilon$, both with $k = 10$. The reason for choosing $k=10$ instead of, e.g., $k=1$ is that from figure \ref{fig:besselFS_epsilon_0005_delta_002_vary_k_01_to_100_mu_0001_to_003_data} we see a slightly higher sensitivity of $\phi$ to noise for approximately $1 < k < 20$ when $d$ starts to increase. So the goal was to capture the worst case scenario for the control stability. For smaller values of $d$ we see as before that a roughly $0.5\%$ change in $f_{1}$ yields about a $5\%$ change in $\phi$. Moreover, the dependence on $\epsilon$ for fixed $d$ is superlinear, consistent with the illposedness of the problem. Interestingly, sensitivity of $\phi$ at $d \approx 0.015$ is better than at nearby values, but of course such a value depends on the other parameters of the problem setup.
 
\begin{figure}
\includegraphics[trim = 10 20 20 0, clip = true, width=\textwidth]{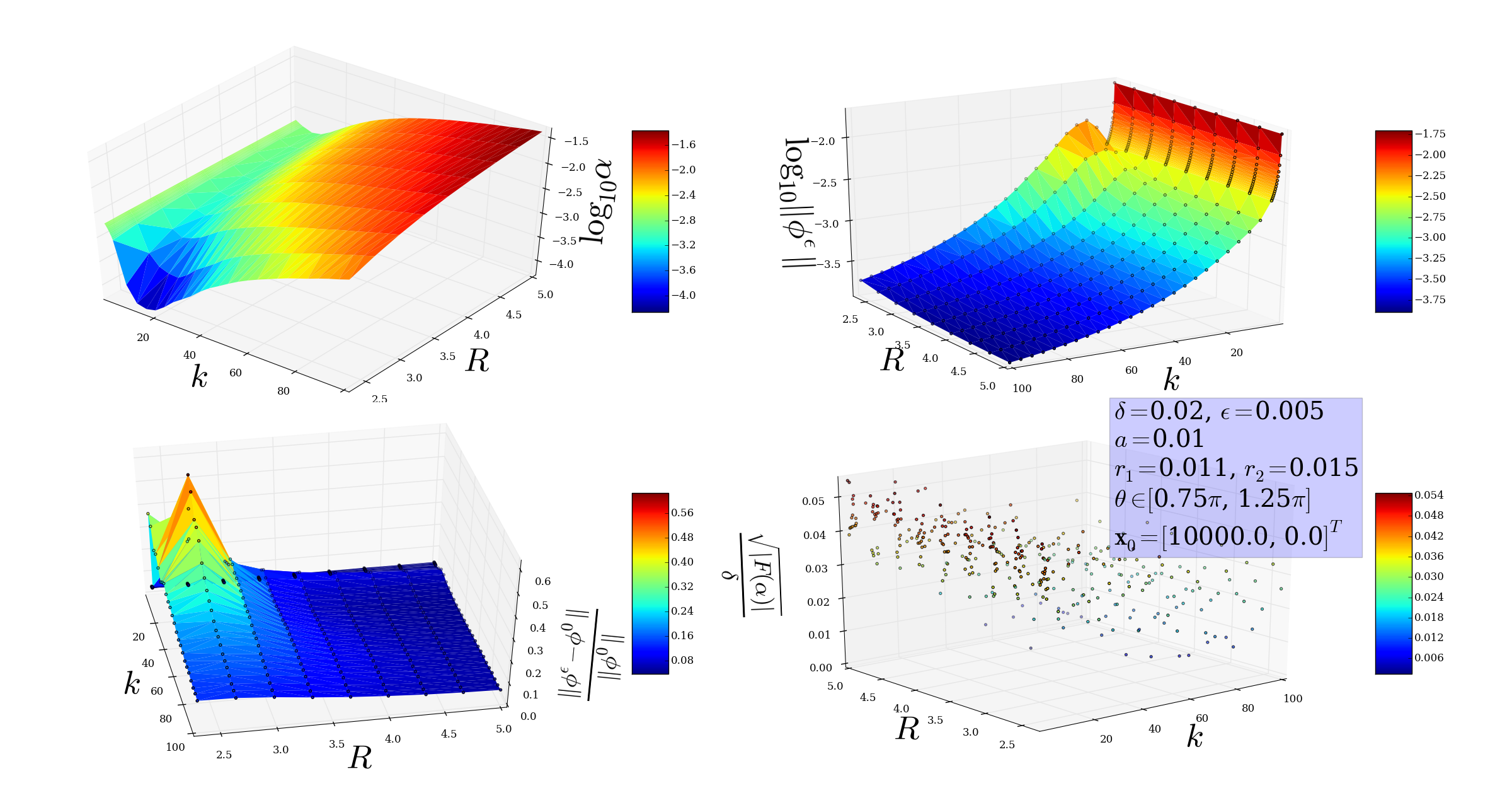}
\caption{Plot vs. $k$ and  $R$ for $f_{1}$ a spherical point source at $\mathbf{x}_{0} = [10000,\, 0]^{T}$.}\label{fig:besselFS_vary_k_R}
\end{figure}
Finally, we consider Figure \ref{fig:besselFS_vary_k_R}, which shows the dependence on $R$ and $k$ for a source at $\mathbf{x}_{0} = [10000,\, 0]^{T}$. Overall, one can see that $R$ can be decreased to around $R = 3$ at any frequency between $0.1$ and $100$ and still achieve the same approximate level of sensitivity for $\phi$ as in the previous plots with $R = 10$.

\section{Conclusions and Future Work}
\label{sec:conclusions}
In this paper we studied the feasibility of the active control scheme for the scalar Helmholtz equation. In the $L^2$ setting, we presented analytic conditional stability results as well as detailed numerical sensitivity studies for the minimal energy solution. We provided several analytic and numerical arguments for the scheme's feasibility and broadband character in the near field when the interrogating field is a far field of a far field observer.

We focused our discussion in this paper only on the case of an interrogating far field point source (i.e. similar to a plane wave with corresponding decay) because we believe that this situation is relevant in usual radar or sonar detection problems. In contrast, the case of an interrogating plane wave corresponds to a different problem, where the observer is close to the source and control region and thus the interrogating signal does not have sufficient decay.

In fact, we have numerically studied the case when the interrogating field is a plane wave without decay or a given uniform field. We observed the scheme does not behave well for the uniform field and that although the stability and accuracy of the near field scheme are essentially independent of the plane wave direction, the overall performance of the scheme is not as good when compared to the case of an interrogating signal coming from a far field observer presented above. In fact, for the same settings as in Figure \ref{fig:nonincreasing_F} when comparing the case of an interrogating far field point source with an interrogating plane wave, we obtained $5\%$ versus $8\%$ stability error and power budget levels of $\approx 10^{-1}$ versus $\approx 10$. We conclude that the scheme performance depends not only on the location of the control region with respect to the source region but also on the amplitude and oscillatory pattern of the incoming field. 

Currently we are considering a more localized basis for $L^2(\partial D_a)$ (e.g. delta function basis, or splines) in order to better observe the field characteristics in the control region $D_c$ and around the antenna $D_a$. We also plan to study the active control scheme for linear arrays and for large elongated antennas. Then, as a next step in our research efforts, we will work on the extension of the current numerical sensitivity study to three dimensions and full Maxwell system and on the study of near field control with planar and conformal arrays.

\bibliographystyle{plain}
\bibliography{controlpaper}

\end{document}